\documentclass[12pt,a4paper]{article}

\usepackage{amsthm}
\usepackage{amsmath}
\usepackage{hyperref}
\usepackage{amssymb}
\usepackage{tikz}
\usepackage{graphicx}

\usetikzlibrary{backgrounds}
\usetikzlibrary{shapes}
\usetikzlibrary{decorations.markings}
\usepackage{multirow}
\usepackage{enumitem}
\hypersetup{
pdftitle={Characterizing Width Two for Variants of Treewidth
}, 
pdfauthor={Hans L. Bodlaender, Vincent J.C. Kreuzen, Stefan Kratsch, O-joung Kwon and Seongmin Ok}
}
\newcommand\abs[1]{\lvert #1\rvert}
\newtheorem{THM}{Theorem}[section]
\newtheorem{LEM}[THM]{Lemma}
\newtheorem*{THMMAIN}{Theorem \ref{theorem:obstruction}}
\newtheorem{COR}[THM]{Corollary}
\newtheorem{PROP}[THM]{Proposition}
\theoremstyle{definition}
\newtheorem{definition}{Definition}[section]
\newtheorem{claim}{Claim}[THM]
\newcommand{\rst}[1]{\ensuremath{{\mathbin\upharpoonright}\raise-.5ex\hbox{$#1$}}}

\newcommand\sctw{\operatorname{sctw}}

\newcommand\tw{\operatorname{tw}}
\newcommand\sptw{\operatorname{spghtw}}
\newcommand\spctw{\operatorname{spctw}}
\newcommand\pw{\operatorname{pw}}

\newcommand\dptw{\operatorname{dspghtw}}
\newcommand\td{\operatorname{td}}
\newcommand\MSO{\operatorname{MSO}}

\title{Characterizing Width Two \\ for Variants of Treewidth}

\author{Hans L. Bodlaender%
\thanks{Department of Computing Science, Utrecht
University, P.O. Box 80.089, 3508 TB, Utrecht, the Netherlands. 
\href{mailto:h.l.bodlaender@uu.nl}{h.l.bodlaender@uu.nl}.}
\and 
Stefan Kratsch
\thanks{TU Berlin, Ernst-Reuter-Platz 7, 10587 Berlin, Germany. Work done while supported by the Netherlands Organization for Scientic Research, N.W.O., project `KERNELS: Combinatorial Analysis of Data Reduction'.
\href{mailto:stefan.kratsch@tu-berlin.de}{stefan.kratsch@tu-berlin.de}.}
\and
Vincent J.C. Kreuzen%
\thanks{School of Business and Economics, Quantitative Economics, Maastricht
University, the Netherlands.
The work was partially done while at Department of Computing Science, Utrecht University. \href{mailto:v.kreuzen@maastrichtuniversity.nl}{v.kreuzen@maastrichtuniversity.nl}}
\and O-joung Kwon%
\thanks{Department of Mathematical Sciences, KAIST, 291 Daehak-ro
  Yuseong-gu Daejeon, 305-701 South Korea.
  This author is supported by Basic Science Research
  Program through the National Research Foundation of Korea (NRF)
  funded by  the Ministry of Science, ICT \& Future Planning
  (2011-0011653).
  \href{mailto:ojoung@kaist.ac.kr}{ojoung@kaist.ac.kr}.}
\and Seongmin Ok%
\thanks{DTU Compute, Technical University of Denmark, DK-2800 Lyngby, Denmark.
\href{mailto:seok@dtu.dk}{seok@dtu.dk}}
}

\begin{document}

\maketitle
\addtocounter{footnote}{1}
\footnotetext{An earlier paper on the characterization of special treewidth two appeared in the proceedings of WG 2013 \cite{BodlaenderKK13}.}

\begin{abstract}
In this paper, we consider the notion of \emph{special treewidth}, recently introduced by Courcelle~\cite{Courcelle2012}. In a special tree decomposition, for each vertex $v$ in a given graph, the bags containing $v$ form a 
rooted path. 
We show that the class of graphs of special treewidth at most two is closed under taking minors, and give the complete list of the six minor obstructions.
As an intermediate result, we prove that every connected graph of special treewidth at most two can be constructed by arranging blocks of special treewidth at most two in a specific tree-like fashion.

Inspired from the notion of special treewidth, we introduce three natural variants of treewidth, namely \emph{spaghetti treewidth}, \emph{strongly chordal treewidth} and \emph{directed spaghetti treewidth}.
All these parameters lie between pathwidth and treewidth, and we provide common structural properties on these parameters.
For each parameter, we prove that the class of graphs having the parameter at most two is minor closed, and
we characterize those classes in terms of a \emph{tree of cycles} with additional conditions.
Finally, we show that for each $k\geq 3$, the class of graphs with special treewidth, spaghetti treewidth, directed spaghetti treewidth, or strongly chordal treewidth,
respectively at most $k$, is not closed under taking minors.
\end{abstract}

\section{Introduction}
\label{section:introduction}
\emph{Treewidth} and \emph{pathwidth} are one of the basic parameters in graph algorithms
and they play an important role in structural graph theory.
Numerous problems which are NP-hard on general graphs, have been shown to be solvable in polynomial time on graphs of bounded treewidth~\cite{ArnborgP89, Bodlaender98}. 
Courcelle~\cite{Courcelle90} provided a celebrated algorithmic meta-theorem which states that every
graph property expressible in monadic second-order logic formulas ($\MSO_2$) can be decided in linear
time on graphs of bounded treewidth.

In this paper, we discuss a relatively new notion of \emph{special treewidth}, introduced by
Courcelle~\cite{Courcelle2012}.
A special tree decomposition is a tree decomposition 
where for each vertex of a given graph, the bags containing this vertex form a rooted path in the tree.
Courcelle developed this parameter to reduce the difficulty in representing tree decompositions algebraically.
The monadic second-order logic ($\MSO_2$) checking algorithm for treewidth in the meta-theorem is based on the constructions of finite automata, and
he observed that these constructions
become much simpler
when working with special tree decompositions compared to standard tree decompositions.

Courcelle asked several questions on properties of special treewidth. 
One of the questions was how to characterize the class of graphs of special treewidth at most $k$ by forbidden configurations.
In this context, he showed that the graphs of special treewidth one are exactly the forests, but if $k\ge 5$, then the class of graphs of special treewidth at most $k$ is not closed under taking minors.

In this paper, we prove that the class of graphs of special treewidth at most two is closed under taking minors, and provide the minor obstruction set.
We also sharpen Coucelle's bound, and show that for $k\ge 3$, the class of graphs of special treewidth at most $k$ is not closed under taking minors.
The graph $K_4$ denotes the complete graph on four vertices, and the other five graphs 
are depicted in Figure~\ref{fig:pw2} and Figure~\ref{figure:forbidden}.

\begin{THMMAIN}
A graph has special treewidth at most two if and only if it has no minor isomorphic to $K_4, D_3, S_3, G_1, G_2,$ or $G_3$.
\end{THMMAIN}
\noindent To show this, we first prove that every block of special treewidth at most two must have pathwidth at most two.
But it is not a sufficient condition for having special treewidth two, and we establish a precise condition how those blocks can be attached to obtain a graph of special treewidth two.

\begin{table}[t]
\begin{center}
\begin{tabular}{|c|c|}
\hline
Parameter & Graph Class \\
\hline
treewidth & chordal graphs \\
pathwidth & interval graphs \\
special treewidth & RDV graphs \\
directed spaghetti treewidth & DV graphs \\
spaghetti treewidth & UV graphs \\
strongly chordal treewidth & strongly chordal graphs \\
treedepth & trivially perfect graphs \\
\hline
\end{tabular}
\caption{Graph parameters which can be defined by the clique number
of a supergraph from a class of graphs. Graph classes are defined in Section~\ref{section:preliminaries}.}
\label{table:graphclasses}
\end{center}
\end{table}

Inspired by special treewidth,
we introduce new three variants of treewidth.
From the results by Courcelle, we observe that
having bounded special treewidth is a 
much stronger property than having bounded treewidth.
We can naturally ask whether there exist elegant width parameters lying between special treewidth and treewidth,
which establish a link from pathwidth to treewidth.

Two variants, \emph{spaghetti treewidth} and \emph{directed spaghetti treewidth},  are defined by taking different models of tree decompositions.
While in the intersection model of special treewidth,
we associate each vertex with a rooted path, 
in a spaghetti tree decomposition, 
the bags containing each vertex form a `usual' path  in a tree (that is, without the condition of being rooted), and 
in a directed spaghetti tree decomposition,
the bags containing each vertex form a directed path in a tree with a given direction.
The \emph{strongly chordal treewidth} of a graph $G$ is defined as the minimum of the clique number of $H$ minus one over
all strongly chordal supergraphs $H$ of $G$. 
These parameters are at most the pathwidth and at least the treewidth of the graph.

Each of these new parameters can be alternatively defined as the minimum of the clique number of all supergraphs where the supergraphs belong to a certain graph class.
Another related notion is treedepth~\cite{BodlaenderDJKKMT98,BodlaenderGHK95}, and it can be defined as the minimum of the clique number of all trivially perfect supergraphs of a given graph.
Table~\ref{table:graphclasses} gives an overview of the parameters and the corresponding classes.

We expect that these new parameters can be used to provide a link between pathwidth and treewidth by yielding new structural or algorithmic results.
As a similar approach,  Fomin, Fraigniaud, and Nisse~\cite{FominFN09} introduced a parameterized variant of tree decompositions, called \emph{$q$-branched tree decompositions}, and provided a unified method to compute pathwidth and treewidth.
In this paper, we study common structural properties of our notions.

For each of the three parameters, we show that the class of graphs having width at most two is closed under taking minors.
Moreover,  we
precisely describe how those graphs look like in terms of {\em trees of cycles} with specific 
conditions depending on the parameter. 
Trees of cycles were used to characterize treewidth two~\cite{BodlaenderK93} and pathwidth two~\cite{BodlaenderF96b}.
In Table~\ref{table:overview}, we see an overview of the different parameters and
the minor obstruction sets for these classes.
As 2-connected graphs play a special role in several proofs, the
2-connected graphs in the obstruction sets are given in the second column.
In addition, for each of these parameters and each value $k\geq 3$, we show that the class of graphs with
the parameter at most $k$ is not closed under taking minors.

\begin{table}
\begin{center}
\begin{tabular}[t]{|c| c |c|}
\hline
Graph classes & Minor obstructions for & Minor obstructions for\\   
& $2$-connected graphs &  general graphs \\
\hline
$\tw \le 2$ & $K_4$ (see~\cite{ArnborgPC90,SatyanarayanaT90}) &  $K_4$ (see~\cite{ArnborgPC90,SatyanarayanaT90})  \\  
\hline
$\sptw \le 2$ & $K_4, D_3$ & $K_4, D_3$ \\  
\hline
$\sctw\le 2$ & $K_4, S_3$ & $K_4, S_3$ \\  
\hline
$\dptw\le 2$ & $K_4, D_3, S_3$ & $K_4, D_3, S_3$ \\  
\hline
$\spctw\le 2$ & $K_4, D_3, S_3$ & $6$ graphs \\  
\hline
$\pw\le 2$ & $K_4, D_3, S_3$~\cite{BJP2012, BodlaenderKK13} &  $110$ graphs~\cite{KinnersleyL94}
\\
\hline
$\td \le 3$ & $K_4, C_5$ \cite{DvorakGT12} & $12$ graphs~\cite{DvorakGT12} \\  
\hline
\end{tabular}
\end{center}
\caption{Summary of results. $\tw$, $\sptw$, $\sctw$, $\dptw$, $\spctw$, $\pw$ 
and $\td$ denote 
treewidth, spaghetti treewidth, strongly chordal treewidth, directed spaghetti treewidth, special treewidth, pathwidth, and treedepth respectively.} \label{table:overview}
\end{table}

Our characterizations in terms of forbidden minors fit in a line of
research, originated by the ground breaking results in the graph
minor project by Robertson and Seymour~\cite{RobertsonS20}. From the
results of Robertson and Seymour, for every minor-closed class $\cal G$ of graphs, there exists a finite obstruction set $ob({\cal G})$ of graphs such that
for each graph $H$, $H\in {\cal G}$ if and only if $H$ has no minor isomorphic to a graph in $ob({\cal G})$. 
For several minor-closed graph classes, the
obstruction set is known, for example, planar graphs ($\{K_5,K_{3,3}\}$ \cite{Wagner37}),
graphs embeddable in the projective plane \cite{Archdeacon83},
graphs of treewidth at most two ($\{K_4\}$, see \cite[Proposition 12.4.2]{Diestel10}),
graphs of treewidth at most three (a set of four graphs \cite{ArnborgPC90}),
graphs of pathwidth at most two (a set of 110 graphs \cite{KinnersleyL94}), and
outerplanar graphs ($\{K_4,K_{2,3}\}$).
The obstruction set of graphs of treedepth at most three (and smaller values)
was given by
Dvo\v{r}\'{a}k, Giannopoulou and Thilikos~\cite{DvorakGT12}; it contains exactly twelve graphs.

This paper is organized as follows. In Section~\ref{section:preliminaries},
we give a number of preliminary definitions and results, including the 
{\em trees of cycles} and {\em paths of cycles} models for $2$-connected graphs
of treewidth and pathwidth two. 
In Section~\ref{section:spaghetti}, we give the characterizations of graphs of spaghetti treewidth at most two.
Section~\ref{section:stronglychordal} discusses graphs with strongly chordal treewidth at most two.
In Section~\ref{section:special}, we discuss graphs of special treewidth at most two, and obtain similar results for graphs of directed spaghetti treewidth at
most two.
Section~\ref{section:width3} considers classes with special treewidth,
strongly chordal treewidth, spaghetti treewidth, or directed spaghetti
treewidth, respectively, at most $k$, for $k\geq 3$. We show that none of
these classes is closed under taking minors.
Some final remarks are made in Section~\ref{section:conclusions}.

\begin{figure}[t]
 \tikzstyle{v}=[circle, draw, solid, fill=black, inner sep=0pt, minimum width=3pt]
 \begin{center}
\begin{tikzpicture}[scale=0.08]

 \node[v](v1) at (0,30) {}; 
 \node[v](v2) at (30,30) {}; 
  \foreach \x in {20,30, 40}
 {
 \node[v] at (10,\x) {}; 
 \node[v] at (20,\x) {};
 \draw (v1) -- (10,\x)--(20,\x)--(v2);
 }
\node at (15,15) {$D_3$}; 

\end{tikzpicture}\qquad
\begin{tikzpicture}[scale=0.06]

 \node[v](v1) at (0,0) {}; 
 \node[v](v2) at (10,15) {}; 
 \node[v](v3) at (20,30) {};
 \node[v](v4) at (30,15) {}; 
 \node[v](v5) at (40,0) {};
 \node[v](v6) at (20,0) {}; 
 \draw (v1) -- (v2)--(v3)--(v4)--(v5)--(v6)--(v1);
 \draw (v2) -- (v4) -- (v6) -- (v2);
\node at (20,-7) {$S_3$}; 
 
\end{tikzpicture}
\end{center}
\caption{The graphs $D_3$ and $S_3$. The graph $S_3$ is called the 3-sun.}\label{fig:pw2}
\end{figure}
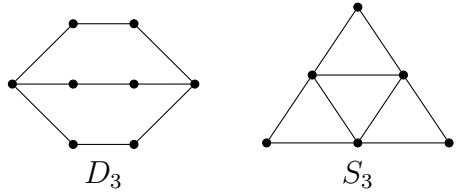
\begin{figure}
\begin{center}
\includegraphics*[viewport=0 90 480 180, scale=0.70]{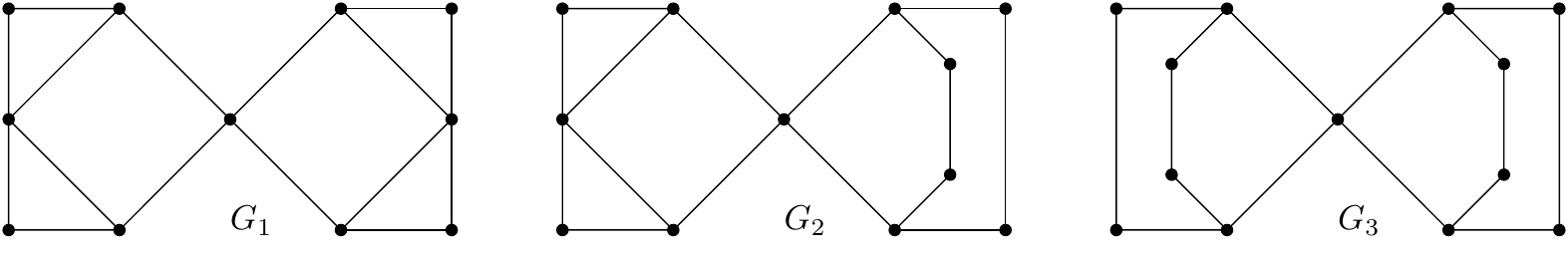}
\end{center}
\caption{The graphs $G_1$, $G_2$, $G_3$ of the minor obstruction set for graphs of special treewidth two, which are not $2$-connected.}
\label{figure:forbidden}
\end{figure}

\section{Preliminaries}
\label{section:preliminaries}

Unless stated otherwise, graphs are considered to be undirected and simple.
Let $G=(V,E)$ be a graph.
For a vertex set $S\subseteq V$, we denote $G[S]$ as the subgraph of $G$ induced on $S$. For $v\in V$ and $e\in E$, we denote $G-v$, $G-e$, $G/e$ as the graphs obtained from $G$ by removing $v$, removing $e$, and contracting $e$, respectively. For a pair of vertices $u, v\in V$ which are not adjacent in $G$, we denote $G+uv$ as the graph obtained from $G$ by adding an edge $uv$.
A subset $S$ of $V$ is a \emph{clique} of $G$ if all vertices in $S$ are pairwise adjacent in $G$.
The \emph{clique number} of a graph, denoted by $\omega(G)$, is defined as the size of a maximum clique in the graph.
A vertex $v$ in a graph $G$ is a \emph{simplicial vertex} if the neighborhood of $v$ forms a clique. 
The \emph{length} of a path is the number of edges in the path.

A graph $G$ is \emph{connected} if  for each pair of vertices $v,w\in V$, there exists a path from $v$ to $w$ in $G$.
A graph $G$ is \emph{$2$-connected}, if $\abs{V}\ge 3$ and $G[V-X]$ is connected for every vertex set $X\subseteq V$ with $\abs{X}\le 1$.
A vertex $v$ of a connected graph $G$ is a \emph{cut vertex} 
if $G-v$ is not connected. 
A {\em block} of a graph $G$ is a maximal connected subgraph of $G$ without a cut vertex.

A graph $H$ is a \emph{minor} of a graph $G$, if $H$ can be obtained from $G$ by a series of 
deletion of a vertex, deletion of an edge, and contraction of an edge. 
A \emph{subdivision} of a graph $G$ is a graph obtained from $G$ by replacing some edges of $G$ with independent paths between their end vertices.

\subsection{Graph Classes}
\label{subsection:definitions}

Several of the notions we look at can be defined in terms of intersection
graphs. 
Let $\mathcal{F}$ be a finite family of graphs. 
The {\em intersection graph}
of $\mathcal{F}$ is the graph $G_{\mathcal{F}}$ whose vertices are the members of the family such that two distinct vertices $f$, $f'$ of $G_{\mathcal{F}}$ are adjacent, if and only if
the corresponding graphs have a common vertex.

A {\em chord} in a cycle of a graph, is a pair of adjacent vertices
on the cycle that are not consecutive on the cycle.
A graph is {\em chordal}, if each cycle with length at least four
has a chord. Alternatively, a graph is a chordal graph, if and only if it is the
intersection graph of subtrees of a tree~\cite{Gavril74}. (See also \cite{BLS2001,Golumbic80}.)
A graph is an \emph{interval} graph if it is the intersection
graph of subpaths of a path.

For three variants of intersection graphs of paths on a tree, we follow the terms in the paper by Monma and Wei~\cite{MonmaW86}.
A graph is an \emph{undirected vertex path graph (shortly, an UV graph)} if it is the intersection graph of a set of paths in a tree.
UV graphs are also known as path graphs~\cite{Gavril78} or VPT graphs~\cite{GolumbicJ85b}.
A \emph{directed tree} is a directed graph whose underlying graph is a tree, and it is called a \emph{rooted tree} if it has exactly one specified vertex called the root and every arc of it is directed to the root.
A graph is a \emph{directed vertex path graph (shortly, a DV graph)} if it is the intersection graph of a set of directed paths in a directed tree.
A graph is a \emph{rooted directed vertex path graph (shortly, an RDV graph)} if it is the 
intersection graph of directed paths in a rooted tree.

A graph $G$ is \emph{strongly chordal} if $G$ is chordal and
every even cycle of length at least six in $G$
has a chord, called an \emph{odd chord}, dividing the cycle into two odd paths of length at least three.

The following relations are well known~\cite{BLS2001}.
	 \begin{align*}
	 &\text{(interval)}\subsetneq \text{(RDV)} \subsetneq \text{(strongly chordal)}\subsetneq \text{(chordal)}, \\
	 &\text{(RDV)}\subsetneq \text{(DV)} \subsetneq \text{(UV)} \subsetneq \text{(chordal)}.
	 \end{align*}

\subsection{Tree Decompositions}

The notions of pathwidth and treewidth were first introduced by Robertson and Seymour \cite{RobertsonS1,RobertsonS2}.

\begin{definition}
A \emph{tree decomposition} of a graph $G=(V,E)$ is 
    a pair $(T,\mathcal{B}=\{B_x\}_{x\in V(T)})$ 
    where $T$ is a tree and for all $x\in V(T)$, $B_x\subseteq V$ which are called \emph{bags}, 	
    satisfying the following three conditions:
\begin{enumerate}
\item[(T1)] $V=\bigcup_{x\in V(T)}B_x$.
\item[(T2)] For every edge $uv$ of $G$, there exists a vertex $x$ of $T$ such that $u$, $v\in B_x$.
\item[(T3)] For every vertex $v$ in $G$, the bags containing $v$ induce a subtree in $T$.
\end{enumerate}
	The \emph{width} of a tree decomposition $(T,\mathcal{B})$ is $\max\{ \abs{B_x}-1:x\in V(T)\}$. 
	The \emph{treewidth} of $G$, denoted by $\tw(G)$, is the minimum width of all tree decompositions of $G$. 
	A \emph{path decomposition} of a graph $G$ is a tree decomposition $(T,\mathcal{B})$ where $T$ is a path. 
	The \emph{pathwidth} of $G$, denoted by $\pw(G)$, is the minimum width of all path decompositions of $G$.
	\end{definition}
	
	We observe the following relations.
	\begin{THM}[folklore; see Bodlaender~\cite{Hans1998}]\label{thm:relclasses}
	Let $k$ be a positive integer.
	\begin{enumerate}
	\item A graph has treewidth at most $k$ if and only if 
	it is a subgraph of a chordal graph with clique number at most $k+1$.
	\item A graph has pathwidth at most $k$ if and only if 
	it is a subgraph of an interval graph with clique number at most $k+1$.
	\end{enumerate}
	\end{THM}

	For a tree decomposition  $\mathcal{I}=(T,\{B_x\}_{x\in V(T)})$ of a graph $G=(V,E)$ and $S\subseteq V$, 
	we denote $P(\mathcal{I}, S)$ as the set of the vertices $x$ in $T$ such that $B_x$ contains a vertex in $S$.
	For $x\in V$, $P(\mathcal{I},\{x\})$ is denoted shortly as $P(\mathcal{I},x)$.

\subsection{Special Treewidth, Directed Spaghetti Treewidth, Spaghetti Treewidth and Strongly Chordal Treewidth}
Courcelle~\cite{Courcelle2012} introduced the notion of special treewidth.
Directed spaghetti treewidth and spaghetti treewidth are natural variants of this
notion.

A \emph{special tree decomposition} of a graph $G$ is 
    a tree decomposition $(T,\mathcal{B})$ where 
    $T$ is a rooted tree and 
    for every vertex $v$ in $G$, the bags containing $v$ induce a directed path in $T$.	The \emph{special treewidth} of $G$, denoted by $\spctw(G)$, is the minimum width of all special tree decompositions of $G$. 	

A \emph{directed spaghetti tree decomposition} of a graph $G$ is 
    a tree decomposition $(T,\mathcal{B})$ where 
    $T$ is a directed tree (not necessarily rooted) and 
    for every vertex $v$ in $G$, the bags containing $v$ induce a directed path in $T$.
    	The \emph{directed spaghetti treewidth} of $G$, denoted by $\dptw(G)$, is the minimum width of all directed spaghetti tree decompositions of $G$. 

A \emph{spaghetti tree decomposition} of a graph $G$ is 
    a tree decomposition $(T,\mathcal{B})$ where
    for every vertex $v$ in $G$, the bags containing $v$ induce a path in $T$.
    The \emph{spaghetti treewidth} of $G$, denoted by $\sptw(G)$, is the minimum width of all spaghetti tree decompositions of $G$. 
	
	From the definitions, we can easily deduce that
	\[\tw(G)\le \sptw(G)\le \dptw(G)\le \spctw(G) \le \pw(G).\]

For a positive integer $k$, we can observe the following from the definitions.
\begin{itemize}
\item A graph has special treewidth at most $k$ if and only if 
	it is a subgraph of an RDV graph with clique number at most $k+1$~\cite{Courcelle2012}.
\item A graph has spaghetti treewidth at most $k$ if and only if 
	it is a subgraph of an UV graph with clique number at most $k+1$.
	\item A graph has directed spaghetti treewidth at most $k$ if and only if 
	it is a subgraph of a DV graph with clique number at most $k+1$.
\end{itemize}

	The \emph{strongly chordal treewidth} of a graph $G$, denoted by $\sctw (G)$, 
	is the minimum $k$ such that
	$G$ is a subgraph of a strongly chordal graph with clique number $k+1$.
	Farber~\cite{Farber82} showed that every RDV graph is strongly chordal. (See also \cite{BLS2001}.) 
	Since a strongly chordal graph is chordal,
	we have that 
	\[\tw(G)\le \sctw(G)\le \spctw(G).\]

\subsection{Models for Treewidth Two and Pathwidth Two}
\label{subsection:treesofcycles}

$2$-connected graphs of treewidth two and of pathwidth two have 
characterizations in terms of {\em trees of cycles} \cite{BodlaenderK93} and 
{\em paths of cycles} \cite{BodlaenderF96b}, respectively.
The {\em cell completion} $\widetilde{G}$ of a $2$-connected graph $G =(V,E)$ is the graph,
obtained from $G$ by adding an edge $vw$ for all pairs
of nonadjacent vertices $v$, $w\in V$ such that $G[V - \{v,w\}]$ has at least three
connected components.

\begin{definition}[Bodlaender and Kloks~\cite{BodlaenderK93}]
\label{def:babette3}
The class of \emph{trees of cycles} is the class of graphs recursively defined as follows.
\begin{itemize}
	\item Each cycle is a tree of cycles.
	\item For each tree of cycles $G$ and each cycle $C$, the graph obtained from $G$ and $C$ by taking the disjoint union and identifying an edge and its end vertices in $G$ with an edge and its end vertices in $C$, is a tree of cycles.
\end{itemize}
\end{definition}
\begin{THM}[Bodlaender and Kloks~\cite{BodlaenderK93}]\label{thm:babette}
Let $G$ be a $2$-connected graph. The graph $G$ has treewidth two if and only
if the cell completion $\widetilde{G}$ of $G$ is a tree of cycles.
\end{THM}

An edge in a tree of cycles $G$ is called an \emph{edge separator} if 
it is contained in at least two distinct chordless cycles of $G$. 
We distinguish two different types of chordless cycles on a tree of cycles.
A triangle of a tree of cycles $G$ is called a \emph{simplicial triangle} if it contains a simplicial vertex; all other chordless cycles are called \emph{body cycles}. 
Every simplicial triangle of a tree of cycles contains at most one edge separators.

\begin{figure}[t]
 \tikzstyle{v}=[circle, draw, solid, fill=black, inner sep=0pt, minimum width=3pt]
 \begin{center}
\begin{tikzpicture}[scale=0.1]

 \node[v](v1) at (10,10) {}; 
 \node[v](v2) at (20,3) {}; 
 \node[v](v3) at (20,20) {}; 
 \node[v](v4) at (30,22) {}; 
 \node[v](v5) at (40,20) {}; 
 \node[v](v6) at (40,3) {}; 
 \node[v](v7) at (55,20) {}; 
 \node[v](v8) at (55,5) {}; 
 \node[v](v9) at (70,20) {}; 
 \node[v](v10) at (70,3) {}; 
 \node[v](v11) at (80,22) {}; 
 \node[v](v12) at (90,10) {}; 
 \node[v](v13) at (80,2) {}; 
 
 \node[v](w1) at (35,28) {}; 
 \node[v](w2) at (47,28) {}; 
 \node[v](w3) at (65,28) {}; 
 
 \draw (v1)--(v2)--(v6)--(v8)--(v10)--(v13)--(v12);
 \draw (v1)--(v3)--(v4)--(v5)--(v7)--(v9)--(v11)--(v12);
 \draw[very thick, gray] (v2)--(v3);\draw[very thick, gray] (v5)--(v6);\draw[very thick, gray] (v5)--(v8);\draw[very thick, gray] (v7)--(v8);
 \draw[very thick, gray] (v9)--(v10);
 \draw (v5)--(w1)--(v6);\draw(v5)--(w2)--(v6);\draw(v9)--(w3)--(v10);

\end{tikzpicture}
\end{center}
\caption{A path of cycles with four simplicial triangles and five edge separators.}\label{fig:mamba}
\end{figure}
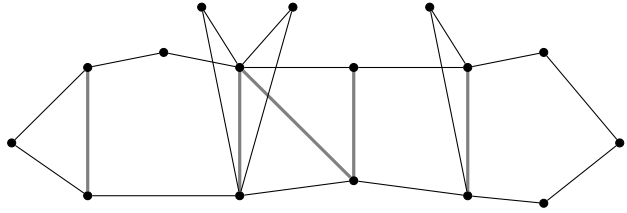

\begin{definition}[Bodlaender and de Fluiter~\cite{BodlaenderF96b}]
\label{def:babette2}
A \emph{path of cycles} is a tree of cycles $G$ for which the following holds.
\begin{enumerate}
	\item Each chordless cycle of $G$ has at most two edge separators.
	\item If an edge $e \in E$ is contained in $m \geq 3$ chordless cycles of $G$, then at least $m - 2$ of these cycles are simplicial triangles.
	\end{enumerate}
\end{definition}

See Figure~\ref{fig:mamba} for an example of a path of cycles.
A triangulated path of cycles has been called 2-caterpillar~\cite{Proskurowski89}. 
Every path of cycles can be represented by a sequence of chordless cycles. This structure will be used to characterize special treewidth two
in Section~\ref{section:special}.

\begin{definition}[Bodlaender and de Fluiter~\cite{BodlaenderF96b}]
\label{def:babetteCP}
Let $G$ be a path of cycles. Let $C = (C_1,\ldots,C_p)$ be a sequence
of chordless cycles such that each chordless cycle in $G$ appears exactly once in the sequence of cycles, and for $1\le i\le p-1$, $C_i$ shares exactly one edge $e_i$ with $C_{i+1}$. 
Let $E = (e_1,\ldots,e_{p-1})$ be the corresponding set of
common edges. The pair $(C,E)$ is called a \emph{cycle path model} for $G$.
\end{definition}

\begin{THM}[Bodlaender and de Fluiter~\cite{BodlaenderF96b}]
\label{thm:babette2}
Let $G$ be a $2$-connected graph. The graph $G$ has pathwidth two
if and only if $\widetilde{G}$ is a path of cycles.
\end{THM}

To obtain similar characterizations for spaghetti treewidth two and 
strongly chordal treewidth two, 
we will observe the structure of trees of cycles with exactly one of the conditions in Definition~\ref{def:babette2}.

\subsection{Simple Cases}
\label{subsection:prelimlemmas}

The main body of our paper discusses the cases where the special treewidth,
spaghetti treewidth, directed spaghetti treewidth or strongly chordal treewidth is at most two. We now briefly discuss the much
simpler case when these parameters are at most one.

	\begin{PROP}\label{prop:trees}
	Let $G$ be a graph. The following are equivalent.
	\begin{enumerate}
	\item $G$ is a forest.
	\item $G$ has treewidth at most one.
	\item $G$ has spaghetti treewidth at most one.
	\item $G$ has strongly chordal treewidth at most one.
	\item $G$ has directed spaghetti treewidth at most one.
	\item $G$ has special treewidth at most one.
	\end{enumerate}	
	\end{PROP}
	 \begin{proof}
Since the treewidth, spaghetti treewidth, special treewidth, directed spaghetti
treewidth or strongly chordal treewidth of a graph
equals the maximum of the parameter over the connected components of a 
graph, it is sufficient to show this proposition for connected graphs.

We assume that $G$ is connected.
It is well known that a connected graph has treewidth at most one, if and only if
it is a tree. Courcelle~\cite{Courcelle2012} has shown that trees have special
treewidth at most one. 
 From the inequalities $\tw(G)\le \sptw(G) \le \dptw(G)\le \spctw(G)$ and 
 $\tw(G)\le \sctw(G) \le \spctw(G)$,
 we conclude that all of the statements are equivalent.
	 \end{proof}

The following result was observed in the case of special treewidth by Courcelle~\cite{Courcelle2012}. The same proof can be used
to obtain this result for other width measures, as shown below.

\begin{LEM}\label{lemma:universal}
Let $G=(V,E)$ be a graph and let $v\in V$ such that $v$ is adjacent to all vertices of $V-\{v\}$ in $G$. Then  
\[ \spctw(G)=\sptw(G)=\dptw(G)=\pw(G)=\pw(G-v)+1.\]
\end{LEM}

\begin{proof}
If we have a special tree decomposition, spaghetti tree decomposition,
directed spaghetti tree decomposition or path decomposition of $G$, we may
assume that all bags contain $v$ because $v$ is adjacent to all vertices of $V-\{v\}$.
All other bags can be deleted.
Such a special tree decomposition, spaghetti tree decomposition, or directed
spaghetti tree decomposition is also a path decomposition, as the bags containing
$v$ form a path. From this observation,
it follows that the first four terms are equal. 

If we take a path decomposition of $G$ with $v$ belonging to each bag of width $k$, we obtain
a path decomposition of $G-v$ of width $k-1$ by removing $v$ from all bags. If
we have a path decomposition of $G-v$, we can obtain one of $G$ by adding $v$ to each bag. This shows that the last two terms are equal.
\end{proof}

\section{Characterizations of Spaghetti Treewidth Two}
\label{section:spaghetti}

    In this section, 
    we characterize the class of graphs of spaghetti treewidth at most two.
    We first define a variant of the trees of cycles to characterize $2$-connected graphs of spaghetti treewidth two.

 	 \begin{definition}
\label{def:chaintree}
A \emph{chain tree of cycles} is a tree of cycles $G=(V,E)$ for which the following holds.
\begin{itemize}
\item If an edge $e \in E$ is contained in $m \geq 3$ chordless cycles of $G$, then at least $m - 2$ of these cycles are simplicial triangles.
\end{itemize}
\end{definition}

 	We show the following theorem. Let $D_3$ be the graph having two specified vertices and 
    three internally vertex-disjoint paths of length three between those vertices;
    see Figure~\ref{fig:pw2}. 
        
 \begin{THM}\label{thm:mainsp}
	Let $G=(V,E)$ be a graph. The following are equivalent.
	\begin{enumerate}
	\item $G$ has spaghetti treewidth at most two.
	\item Each block of $G$ is either a $2$-connected subgraph whose cell completion is a chain tree of cycles,
	or a single edge, or an isolated vertex. 	
	\item $G$ has no minor isomorphic to $K_4$ or $D_3$.
	\end{enumerate}
\end{THM}

	We need the following lemma.
		
	\begin{LEM}\label{lem:h3minor}
The spaghetti treewidth of a subdivision of $D_3$ is three.
\end{LEM}
\begin{proof}
	Let $H$ be a subdivision of $D_3$ consisting of two vertices, say $a$ and $b$, 
	of degree three and three independent paths $a u_1 u_2 \ldots u_l b$, $a v_1 v_2 \ldots v_m b$ and $a w_1 w_2 \ldots w_n b$, 
	where $l,m,n \geq 2$.
	Let $U=\{u_i:1\le i\le l\}$, $V=\{v_i:1\le i\le m\}$, $W=\{w_i:1\le i\le n\}$.
	It is easy to see that $\sptw(H) \leq 3$. 
	We shall prove $\sptw(H) \ge 3$.

	Suppose that $H$ has a spaghetti tree decomposition $\mathcal{I}=(T, \{B_x\}_{x \in V(T)})$ of width two. 
	We may assume that none of $B_x$ is a singleton. 
	
	If $P(\mathcal{I},a)\cap P(\mathcal{I},b) = \emptyset$, then
	there exists $y\in P(\mathcal{I},a)$ which separates $P(\mathcal{I},a)- \{y\}$ from $P(\mathcal{I},b)$ in $T$. 
	Since $u_1$ is contained in some bag of $P(\mathcal{I},a)$ and $u_l$ is contained in some bag of $P(\mathcal{I},b)$, 
	 $B_y$ contains some $u_i$. 
	 Similarly, $B_y$ contains some $v_j$ and $w_k$. As $a\in B_y$, the size of $B_y$ is at least four. It contradicts to that the width of $\mathcal{I}$ is two.

	Now we assume that $P(\mathcal{I},a)\cap P(\mathcal{I},b) \neq \emptyset$. 
	Note that $P(\mathcal{I},a)\cap P(\mathcal{I},b)$ forms a path in $T$.
	By the same reason as above, 
	\[U' := (P(\mathcal{I},a)\cap P(\mathcal{I},b)) \cap P(\mathcal{I},U) \neq \emptyset.\] 
	Similarly, we obtain
	\[V' := (P(\mathcal{I},a)\cap P(\mathcal{I},b)) \cap P(\mathcal{I},V)\neq \emptyset,\] 
	\[W' := (P(\mathcal{I},a)\cap P(\mathcal{I},b)) \cap P(\mathcal{I},W)\neq \emptyset.\] 
	Since every bag of $P(\mathcal{I},a)\cap P(\mathcal{I},b)$ has two vertices $a$ and $b$, no two of $U', V', W'$ have a common vertex.
	We may assume that $V'$ lies between $U'$ and $W'$ in $P(\mathcal{I},a)\cap P(\mathcal{I},b)$.
	Since $H[V]$ has at least one edge, 
	we must have $P(\mathcal{I},V)- V'\neq \emptyset$.
	We choose $x\in V'$ and $y\in P(\mathcal{I},V)- V'$ such that they are neighbors in $T$.
	Then $\abs{B_x \cap B_y}\leq 1$.
 Since $\abs{B_y}\ge 2$, it contradicts to the $2$-connectedness of $H$. 
	\end{proof}

	The following lemma is a key lemma to obtain a subdivision of $D_3$ as a subgraph.
	
	  \begin{LEM}\label{lem:twocompos}
    Let $k\ge 1$. Let $G=(V,E)$ be a $2$-connected graph and let $u, v\in V$.
    If $\widetilde{G}[V-\{ u, v\}]$ has $k$ components of size at least two,
    then $G$ has $k$ internally vertex-disjoint paths of length at least three from $u$ to $v$.
    \end{LEM}
 	\begin{proof}
	We claim that if $\widetilde{G}[V-\{ u, v\}]$ has a component $H$ such that $\abs{V(H)}\ge 2$,
     then $G[V(H)]$ has an edge.
     Suppose that $G[V(H)]$ has no edges. 
    Since $\abs{V(H)}\ge 2$, there are two vertices $x$ and $y$ of $G[V(H)]$ such that $xy\in E(\widetilde{G})- E$.
	From the definition of a cell completion,
	$G[V-\{x,y\}]$ has at least three components.
	It leads a contradiction because every vertex in $V- \{x, y\}$ is connected to $u$ and $v$ in $G[V-\{x,y\}]$.
	
 	We assume that $\widetilde{G}[V-\{u,v\}]$ has $k$ components $G_1$, $G_2, \ldots, G_k$ where each $G_i$ has at least two vertices.
	By the claim, 
	each $G[V(G_i)]$ has a component $G_i'$ having at least one edge.
	Since $G$ is $2$-connected, 
	$G$ has $k$ internally vertex-disjoint paths of length at least three from $u$ to $v$ along each $G_1'$, $G_2', \ldots, G_k'$, as required.
     \end{proof}

\subsection{Characterization with Cycle Model}

We characterize $2$-connected graphs of spaghetti treewidth two in terms of trees of cycles.
 	
    \begin{THM}\label{thm:main}
	Let $G=(V,E)$ be a $2$-connected graph. Then  
	$G$ has spaghetti treewidth two if and only if 
	the cell completion $\widetilde{G}$ of $G$ is a chain tree of cycles. 	
\end{THM}

	We first show that if a $2$-connected graph $G$ has spaghetti treewidth two,
    then $\widetilde{G}$ is a chain tree of cycles.
    	
    \begin{PROP}\label{prop:main3}
    Let $G=(V,E)$ be a $2$-connected graph.
    If $G$ has spaghetti treewidth two,
    then $\widetilde{G}$ is a chain tree of cycles.
    \end{PROP}

	\begin{proof}
	Suppose $\sptw(G)=2$.
	Since $G$ is $2$-connected and $\tw(G)\le 2$,
	by Theorem~\ref{thm:babette}, $\widetilde{G}$ is a tree of cycles.
	So, it is sufficient to check that 
	every edge separator of $\widetilde{G}$ is contained in at most two body cycles of $\widetilde{G}$. 
    
	Suppose that an edge separator $uv$ is contained in three body cycles in $\widetilde{G}$.
	So, $\widetilde{G}[V-\{u,v\}]$ has three components having at least two vertices.
	By Lemma~\ref{lem:twocompos}, 
	$G$ has three internally vertex-disjoint paths of length at least three from $u$ to $v$, and
	therefore, $G$ has a subgraph isomorphic to a subdivision of $D_3$.
	By Lemma~\ref{lem:h3minor}, $\sptw(G)\ge 3$, which contradicts to the assumption. 
	\end{proof}

	We prove the other direction.
	\begin{PROP}\label{prop:chaintree2}
	Every chain tree of cycles has spaghetti treewidth two.
	\end{PROP}

	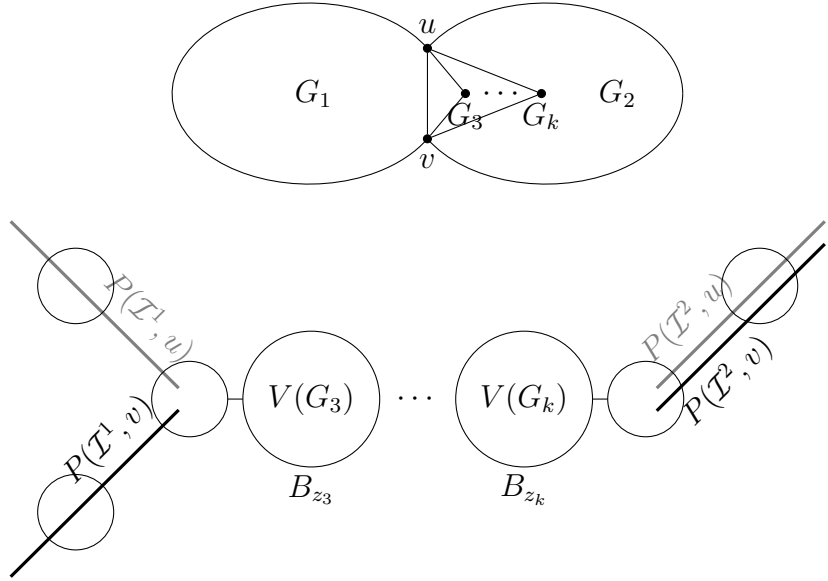
\begin{figure}[t]\centering
 \tikzstyle{v}=[circle, draw, solid, fill=black, inner sep=0pt, minimum width=3pt]
 \begin{tikzpicture}[scale=0.1]
 \draw (15,30) arc (30:330:18 and 12);
 \draw (15,18) arc (210:360:18 and 12);
 \draw (15,30) arc (150:0:18 and 12);  
 \node[v](v1) at (15,30) {}; 
 \node[v](v2) at (20,24) {}; 
 \node[v](v3) at (15,18) {};
 \node[v](v4) at (30,24) {}; 
 \draw (v1) -- (v2)--(v3)--(v1);
 \draw (v1) -- (v4)--(v3);
  \node at (15,33) {$u$}; 
\node at (15,15) {$v$}; 
\node at (20,21) {$G_3$}; 
\node at (30,21) {$G_k$}; 
\node at (0,24) {$G_1$}; 
\node at (40,24) {$G_2$}; 
\node at (25,24) {$\cdots$};

\end{tikzpicture}

\begin{tikzpicture}[scale=0.1]

 \node(v1) at (-10,40) {}; 
 \node(v2) at (15,15) {}; 
 \node(v3) at (-10,-10) {};
 
 \node(v4) at (100,40) {}; 
 \node(v5) at (75,15) {}; 
 \node(v6) at (100,37) {};
 \node(v7) at (75,12) {};

 \draw[gray, very thick] (v1) -- (v2) {node[sloped,above,pos=0.7]{$P(\mathcal{I}^1,u)$}};
 \draw[black, very thick] (v2) -- (v3) {node[sloped,above,pos=0.3]{$P(\mathcal{I}^1,v)$}};
 \draw[gray, very thick] (v4) -- (v5) {node[sloped,above,pos=0.7]{$P(\mathcal{I}^2,u)$}};
 \draw[black, very thick] (v6) -- (v7) {node[sloped,below,pos=0.7]{$P(\mathcal{I}^2,v)$}};
  
\draw (0,30) circle (5);
\draw (15,15) circle (5);
\draw (0,0) circle (5);
\draw (31,15) circle (9);

 \node at (45,15) {$\cdots$}; 
\node at (31,15) {$V(G_3)$}; 
\node at (59,15) {$V(G_k)$}; 
\node at (31,3) {$B_{z_3}$};
\node at (59,3) {$B_{z_k}$};

\draw (59,15) circle (9);
\draw (75,15) circle (5);
\draw (90,30) circle (5);
 
\draw (20,15) -- (22,15);
\draw (70,15) -- (68,15);

\end{tikzpicture}
\caption{The induction step in Proposition~\ref{prop:chaintree2}.}\label{fig:merging}
\end{figure}

	For a tree of cycles $G$, 
	let $\mathcal{G}(G)$ be the set of body cycles of $G$ and let $\mathcal{D}(G)$ be the set of simplicial triangles of $G$.
	A subset $\mathcal{P}$ of $\mathcal{D}(G)$ is called a \emph{potential set} of $G$ if 
	each edge separator of $G$ is contained in exactly two cycles of $\mathcal{G}(G)\cup \mathcal{P}$.
	 Briefly, we will show that 
	 a chain tree of cycles $G$ with a fixed potential set $\mathcal{P}$ admits a special type of a spaghetti tree decomposition.
	For a potential set $\mathcal{P}$ of $G$,
	let $F(G, \mathcal{P})$ be the set of all non-edge separator edges contained in cycles of $\mathcal{G}(G)\cup \mathcal{P}$.

		 \begin{LEM}\label{lem:potential}
	 Let $G=(V,E)$ be a chain tree of cycles  with a potential set $\mathcal{P}$, and let $uv$ be an edge separator of $G$. 
	 Let $H$ be a component of $G[V-\{u,v\}]$ such that  $H':=G[V(H)\cup \{u, v\}]$ and $H'$ is not a chordless cycle. 
	 Let $C'$ be the chordless cycle of $H'$ containing the edge $uv$.
	 Then 
	 \begin{enumerate}
	 \item $\{C : C\in \mathcal{P}, V(C)\subseteq V(H')\}\cup \{C'\}$ is a potential set of $H'$ if $C'$ is a triangle having two edge separators in $G$, and
	 \item $\{C : C\in \mathcal{P}, V(C)\subseteq V(H')\}$ is a potential set of $H'$ if otherwise.
	 \end{enumerate}
	 \end{LEM}
	 \begin{proof}
	Since $H'$ is not a chordless cycle, $C'$ has at least two edge separators in $G$.
	 If either $C'$ is a cycle of length at least $4$ or $C'$ has three edge separators in $G$,
	  then $C'$ is still a body cycle of $H'$, 
	  and there is nothing to prove. 
	  If $C'$ is a triangle having only one edge separator $f$ other than $uv$  in $G$,
	 then $f$ is contained in exactly one cycle of $\{C : C\in \mathcal{P}, V(C)\subseteq V(H')\}$, and $C'$ is a simplicial triangle of $H'$.
	 Thus, $\{C : C\in \mathcal{P}, V(C)\subseteq V(H')\}\cup \{C'\}$ is a potential set of $H'$.
	 \end{proof}

	\begin{proof}[Proof of Proposition~\ref{prop:chaintree2}]
	Let $G=(V,E)$ be a chain tree of cycles.
	Since $G$ is $2$-connected,
	it is sufficient to show that $\sptw(G)\le 2$.
	Let $\mathcal{P}$ be a potential set of $G$.
	We claim that 
	$G$ has a spaghetti tree decomposition $\mathcal{I}=(T,\{B_x\}_{x\in V(T)})$ of width two 
	such that
	\begin{enumerate}
	\item[(GC)] there exists an injective function $g$ from $F(G, \mathcal{P})$ to $V(T)$ where
	 $P(\mathcal{I},u)$ and $P(\mathcal{I},v)$ have a common end vertex  $g(uv)$ in $T$ for $uv\in F(G, \mathcal{P})$.
	 \end{enumerate}
	We prove it by induction on the number of edge separators of $G$.

	If $G$ has no edge separators, then $G$ is a chordless cycle. Let $G$ be a chordless cycle $c_1c_2c_3 \cdots c_mc_1$ for some $m\ge 3$. 
	We construct a tree decomposition $\mathcal{I}=(T,\{B_x\}_{x\in V(T)})$ of $G$ and define a function $g$ from $E$ to $V(T)$ such that
	\begin{itemize}[label={--}]
	\item $T$ is a path $p_0p_1p_2p_3 \cdots p_{m-1}$,
	\item $B_{p_0}=\{c_m, c_1\}$, $B_{p_{m-1}}=\{c_m,c_{m-1}\}$, 
	\item for each $1\le i\le m-2$, $B_{p_i}=\{c_m,c_i,c_{i+1}\}$, and
	\item $g$ is the function from $E$ to $V(T)$ such that
	$g(c_mc_1)=p_0$ and $g(c_ic_{i+1})=p_i$ for all $1\le i\le m-1$.
	\end{itemize}
	We can easily check that 
	$\mathcal{I}$ is a spaghetti tree decomposition of $G$ having width two and
	it satisfies the condition (GC).

	Now suppose that $G$ has an edge separator $uv$.
	Let $H_1, H_2, \ldots, H_k$ be the components of $G[V-\{u,v\}]$ where $k\ge 2$ and
	let $G_i=G[V(H_i)\cup \{u, v\}]$.
	Note that exactly two graphs of $\{G_i\}_{1\le i\le k}$ have a cycle in 
	$\mathcal{G}(G)\cup \mathcal{P}$.
	Without loss of generality, 
	we assume that $G_1$ and $G_2$ have a cycle in 
	$\mathcal{G}(G)\cup \mathcal{P}$.

	We first check that for each $j\in \{1,2\}$, $G_j$ admits a spaghetti tree decomposition satisfying the condition (GC).
	We may assume that $G_j$ is not a chordless cycle.
	For each $j\in \{1,2\}$, let 
	$\mathcal{P}_{j}=\{C : C\in \mathcal{P}, V(C)\subseteq V(G_j)\}$, and
	let $C_j$ be the chordless cycle of $G_j$ containing the edge $uv$.
	We define 
	\begin{itemize}
	 \item $\mathcal{P}'_j:=\mathcal{P}_{j}\cup \{C_j\}$ if $C_j$ is a triangle having two edge separators in $G$,
	 \item $\mathcal{P}'_j:=\mathcal{P}_{j}$ if otherwise.
	\end{itemize}
	In both cases, by Lemma~\ref{lem:potential},
$P_{j}'$ is a potential set of $G_j$
	and $C_j$ is contained in $\mathcal{G}(G_j)\cup \mathcal{P}_j'$.
	By the induction hypothesis, 
	$G_j$ has a spaghetti tree decomposition $\mathcal{I}^j=(T^j,\{B^j_x\}_{x\in V(T^j)})$ of width two such that
	there is an injective function $g_j$ from $F(G_j, \mathcal{P}'_j)$ to $V(T^j)$ which satisfies the condition (GC).
	Let $w_j=g_j(uv)$.

	We construct a new tree decomposition $(T,\{B_x\}_{x\in V(T)})$ and define the function $g$ from $F(G, \mathcal{P})$ to $V(T)$ such that
	\begin{itemize}[label={--}]
	\item $T$ is obtained from the disjoint union of $T^1$, $T^2$ and the path $z_3\cdots z_k$ by adding edges $z_3w_1$, $z_kw_2$, and
	\item for all $1\le i\le 2$ and $x\in V(T_i)$, $B_x=B^i_x$,
	\item for all $3\le i\le k$, $B_{z_i}=V(G_i)$, and
	\item $g(e)=g_i(e)$ if $e\in F(G_i, \mathcal{P}'_i)$ for $i\in \{1,2\}$.
	\end{itemize}
	This case is depicted in Figure~\ref{fig:merging}. 
	Clearly, $P(\mathcal{I},u)$ and $P(\mathcal{I},v)$ form paths in $T$.
	So, $(T,\{B_x\}_{x\in V(T)})$ is a spaghetti tree decomposition of $G$ having width two.
	Because the only $P(\mathcal{I}^j,u)$ and $P(\mathcal{I}^j,v)$ are changed in each $T^j$ and $uv\notin F(G, \mathcal{P})$, 
	$g$ is injective, as required.
	\end{proof}

	\begin{proof}[Proof of Theorem~\ref{thm:main}]
	If $\sptw(G)=2$, then by Proposition~\ref{prop:main3},
	$\widetilde{G}$ is a chain tree of cycles.
	If $\widetilde{G}$ is a chain tree of cycles,
	then by Proposition~\ref{prop:chaintree2},
	$\sptw(\widetilde{G})=2$.
	Since $G$ is $2$-connected and spaghetti treewidth does not increase when taking a subgraph,
	$\sptw(G)=2$.
	\end{proof}

\subsection{The Minor Obstruction Set for Spaghetti Treewidth Two}
	We provide the minor obstruction set for the class of $2$-connected graphs of spaghetti treewidth two. 
	
\begin{THM}\label{thm:sptwmo}
	Let $G=(V,E)$ be a $2$-connected graph. 
	The graph $G$ has spaghetti treewidth two
	if and only if it has no minor isomorphic to $K_4$ or $D_3$.
\end{THM}
	\begin{proof}
	Suppose $G$ has spaghetti treewidth two.
	If $G$ has a minor isomorphic to $K_4$, 
	then $\sptw(G)\ge \tw(G)\ge 3$. 
	If $G$ has a minor isomorphic to $D_3$,
	then $G$ has a subgraph isomorphic to a subdivision of $D_3$. 
	By Lemma~\ref{lem:h3minor},
	$\sptw(G)\ge 3$ and it contradicts to our assumption on $G$. 
	
	Now suppose that $G$ has no minor isomorphic to $K_4$ and $\sptw(G)\ge 3$.
	Since $G$ is $2$-connected, by Theorem~\ref{thm:babette} and~\ref{thm:main},
	$\widetilde{G}$ is a tree of cycles but not a chain tree of cycles. 
	So, $\widetilde{G}$ has an edge separator $uv$ such that
	$uv$ is contained in three body cycles.
	Therefore, $\widetilde{G}[V-\{u,v\}]$ has three components having at least two vertices, 
	and by Lemma~\ref{lem:twocompos}, 
	$G$ has three internally vertex-disjoint paths of length at least three from $u$ to $v$.
	Thus, $G$ has a minor isomorphic to $D_3$.
	\end{proof}
	
	Using the following lemma, we have the results for general cases.
	\begin{LEM}\label{lem:sptwblock}
	Let $k$ be a positive integer. 
	A graph has spaghetti treewidth at most $k$
	if and only if every block of it has spaghetti treewidth at most $k$.
	\end{LEM}
	\begin{proof}
	The forward direction is trivial. 
	For the converse direction, suppose every block of a graph $G$ has spaghetti treewidth at most $k$.
	We may assume that $G$ is connected.
	We prove by induction on the number of cut vertices of $G$.
	We may assume that $G$ has a cut vertex $v$.
	Let $H_1$, $H_2, \ldots, H_k$ be the components of $G-v$ and let $G_i=G[V(H_i)\cup \{v\}]$.
	By the induction hypothesis, there exists a spaghetti tree decomposition $\mathcal{I}_i$ of $H_i$ having width at most $k$.
	We can obtain a new tree decomposition $\mathcal{I}$ from the disjoint union of the decompositions $\mathcal{I}_i$ by just connecting bags among the bags in $\bigcup_{1\le i\le k} P(\mathcal{I}_i, v)$
	so that $P(\mathcal{I}, v)$ forms a path. It shows that $G$ has spaghetti treewidth at most $k$.
	\end{proof}
	\begin{proof}[Proof of Theorem~\ref{thm:mainsp}]
	By Theorem~\ref{thm:main}, (1) implies (2), and with Lemma~\ref{lem:sptwblock}, (2) also implies (1).
	By Theorem~\ref{thm:sptwmo} and Lemma~\ref{lem:sptwblock}, 
	(1) and (3) are equivalent.
	\end{proof}

\section{Characterizations of Strongly Chordal Treewidth Two}
\label{section:stronglychordal}

	In this section, 
    we characterize the class of graphs of strongly chordal treewidth at most two with cycle model and we provide the minor obstruction set for the class. 
    We introduce another variant of a tree of cycles, called a \emph{tree of two-boundaried cycles}.
    The name `two-boundaried' comes from the property that every chordless cycle of it may attach with other chordless cycles on at most two edges.
    
    \begin{definition}
\label{def:treeofthc}
A \emph{tree of two-boundaried cycles} is a tree of cycles $G$ for which the following holds. 
\begin{itemize}
\item Each chordless cycle of $G$ has at most two edge separators.
\end{itemize}
\end{definition}

	We mainly show the following. The graph $S_3$ is depicted in Figure~\ref{fig:pw2}. 
   
   	\begin{THM}\label{thm:mainsc}
	Let $G=(V,E)$ be a graph. The following are equivalent.
	\begin{enumerate}
	\item $G$ has strongly chordal treewidth at most two.
	\item Each block of $G$ is either a $2$-connected subgraph whose cell completion is a tree of two-boundaried cycles, or a single edge, or an isolated vertex.
	\item $G$ has no minor isomorphic to $K_4$ or $S_3$. 	
	\end{enumerate}
\end{THM}

    Unlike $D_3$, 
    it seems to be tedious to characterize the subgraph minimal graphs containing $S_3$ as a minor, because $S_3$ has maximum degree four.
    So, we first show that the class of graphs of strongly chordal treewidth at most two is closed under taking minors.

	We use the following fact that $S_3$ has strongly chordal treewidth three. 
    
     \begin{LEM}\label{lem:sctws3}
   The strongly chordal treewidth of $S_3$ is three.
	\end{LEM}
	\begin{proof}
	If we add one odd chord in the cycle of length six in $S_3$,
	then the resulting graph is a strongly chordal graph with clique number four.
	Therefore, $\sctw(S_3)\le 3$. 
	If there exists a strongly chordal graph $H$
	having $S_3$ as a subgraph,
	then the cycle of length six in $H[V(S_3)]$ must have an odd chord.
	Thus, $\omega (H)\ge 4$ and it implies that $\sctw(S_3)\ge 3$.
		 \end{proof}

	The following lemma will be used to find $S_3$ as a minor.
    \begin{LEM}\label{lem:avoidingcycle}
    Let $G=(V,E)$ be a $2$-connected graph having treewidth two and let $uv$ be an edge separator of $\widetilde{G}$.
 	If $C$ is a chordless cycle of $\widetilde{G}$ containing $uv$, 
 	then $G$ has two internally vertex-disjoint paths from $u$ to $v$ such that they have no vertices of $C$ except $u$ and $v$.
	\end{LEM}
	\begin{proof}
	We have two cases.
	
	\smallskip
	\noindent\emph{Case 1. $uv\in E$.} 
	The edge $uv$ is one of the required paths.
	Since $uv$ is an edge separator of $\widetilde{G}$, 
	$\widetilde{G}[V-\{u,v\}]$ has at least one component having no vertices of $C$. 
	Thus, $G[V-\{u,v\}]$ has a component $H$ having no vertices of $C$, and $G$ has a path from $u$ to $v$ in $G$ along $H$.
	
	\smallskip
	\noindent\emph{Case 2. $uv\in E(\widetilde{G})- E$.} 
	By the definition of a cell completion, $G[V-\{u,v\}]$ has at least three components.
	Therefore, $G[V-\{u,v\}]$ has two components $H_1$ and $H_2$ which contain no vertices in $C$.
	Clearly, there are two internally vertex-disjoint paths from $u$ to $v$ in $G$ along $H_1$ and $H_2$.
	\end{proof}

\subsection{Contractions on Graphs of Strongly Chordal Treewidth Two}
	We show that the class of graphs of strongly chordal treewidth at most two is closed under taking minors. 
\begin{PROP}\label{prop:contraction}
	The class of graphs of strongly chordal treewidth at most two is closed under taking minors.
	\end{PROP}

    We will use a known characterization of strongly chordal graphs.
	For an integer $n\ge 3$, a graph $G$ is called an \emph{$n$-sun} 
	if $G$ is a graph with $2n$ vertices
	which are partitioned into two parts $U = \{u_1, u_2, \ldots, u_n\}$ and $W = \{w_1, w_2, \ldots, w_n\}$ 
	such that $U$ induces a clique, $W$ induces an independent set and 
	each vertex $w_i$ in $W$ is adjacent to $u_j$
	if and only if $i\equiv j$ or $i\equiv j+1$ (mod $n$). 
	Note that $S_3$ is the $3$-sun.

	\begin{THM}[Farber~\cite{Farber83}]\label{thm:charsc}
	A graph $G$ is strongly chordal if and only if 
	$G$ is chordal and it has no induced subgraph isomorphic to a sun.
	\end{THM}

	Since taking subgraphs does not increase strongly chordal treewidth, it is enough to show the following.
	
	 \begin{PROP}\label{prop:contractionsc}
	Let $G=(V,E)$ be a strongly chordal graph of $\omega(G)\le 3$ and let $e\in E$.
	Then $G/e$ is strongly chordal.
	\end{PROP}		
	
	We will prove, by induction on the size of odd cycles $C$ in $G$ which contain the contracted edge $e$, that $C/e$ has an odd chord. For the base case, we need a lemma.
	
	\begin{LEM}\label{lem:basiccase}
	Let $G=(V,E)$ be a chordal graph and let $e\in E$. 
	If $G/e=S_3$, then either $G$ is not strongly chordal or $\omega(G)=4$.
	\end{LEM}
	\begin{proof}
	Let $V=\{w,v_1, v_2, \ldots, v_6\}$ and
	let us assume that for some $i\in \{1, 2, \ldots, 6\}$, 
	$e=wv_i$, and 
	after contracting $wv_i$ in $G$, the contracted vertex is again labeled by $v_i$.
	Suppose $G/e=S_3$ where $v_1v_2 \cdots v_6v_1$ is the cycle of length six and 
	$v_2v_4v_6$ is the triangle in the middle.
	By symmetry, we may assume $i=1$ or $2$.
	We may also assume that both $w$ and $v_i$ have degree at least two in $G$, otherwise one of $G-w$ and $G- v_i$ is isomorphic to $S_3$.
	
	If $i=1$,
	then the number of edges between $\{w,v_1\}$ and $\{v_2,v_6\}$ is at least three because $G$ is chordal.
	So, one of $G- w$ and $G- v_1$ must be isomorphic to $S_3$.
	Therefore, $G$ is not strongly chordal.
	
	Now we assume that $i=2$.
	We first claim that each of $w$ and $v_2$ is adjacent to exactly one of $v_1$ and $v_3$, and the neighbors of $w$ and $v_2$ are distinct.
	If $v_2$ is adjacent to both $v_1$ and $v_3$,
	then $v_1v_2v_3v_4v_5v_6v_1$ is an even cycle of length six  in $G$ without chords $v_1v_4, v_2v_5, v_3v_6$. 
	So, $G$ is not strongly chordal.
	By the same reason, $w$ cannot be adjacent to both $v_1$ and $v_3$.
	If $v_2$ is adjacent to neither $v_1$ nor $v_3$,
	then $w$ must be adjacent to both $v_1$ and $v_3$,
	and we already see that it is impossible.
	Thus, each of $w$ and $v_2$ is adjacent to exactly one of $v_1$ and $v_3$ and the neighbors are distinct.
	By symmetry, 
	we may assume that $w$ and $v_2$ are adjacent to $v_1$ and $v_3$, respectively.

		Since $G$ is chordal and $v_1$ is only adjacent to $w$ and $v_6$ which form a cycle with other vertices in $G$, $wv_6$ must be an edge of $G$. By the same reason, $v_2v_4\in E$.
		Since $wv_2v_4v_6$ is a cycle of length $4$, we have either $v_2v_6\in E$ or $wv_4\in E$.
		If $v_2v_6\notin E$ (or $wv_4\notin E$),
	then $G- v_3$ (or $G- v_1$) is isomorphic to $S_3$,
	and therefore $G$ is not strongly chordal.
		If $v_2v_6, wv_4\in E$, then $w(G)=4$, as required.
\end{proof}

	\begin{proof}[Proof of Proposition~\ref{prop:contractionsc}]
	Note that the cycles affected by
	the contraction of $e$ are the cycles containing $e$. 
	As $G/e$ is again chordal, we shall only consider the odd cycles in $G$ of length at least seven which contain $e$. 
	Let $C$ be one of such cycles. We shall show, by induction on the length of $C$, 
	that 
	\begin{itemize}[label={--}]
	\item for every edge $e\in E$, $C/e$ has an odd chord in $G/e$. 
	\end{itemize}
	
	Suppose the length of $C$ is seven.
	Since $G/e$ is chordal, if $C/e$ has no odd chord, then $V(C/e)$ induces a graph isomorphic to $S_3$ in $G/e$. 
	By Lemma~\ref{lem:basiccase}, 
	either $G[V(C)]$ is not strongly chordal or $\omega(G[V(C)])=4$,
	contradicting to our assumption on $G$.
	Thus, $C/e$ must have an odd chord.

	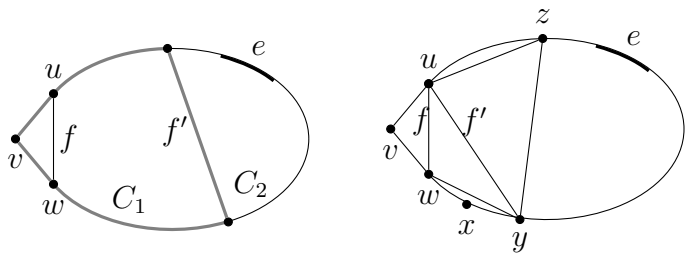
\begin{figure}[t]\centering
 \tikzstyle{v}=[circle, draw, solid, fill=black, inner sep=0pt, minimum width=3pt]
 \begin{tikzpicture}[scale=0.1]
 \draw (15,18) arc (210:360:18 and 12);
 \draw (15,30) arc (150:0:18 and 12);  
 \draw[very thick] (37,35) arc (60:30:19 and 9);  

  \draw[very thick, gray] (15,18) arc (210:294:18 and 12);
 \draw[very thick, gray] (15,30) arc (150:90:18 and 12);  
   
 \node[v](v1) at (15,30) {}; 
 \node[v](v2) at (10,24) {}; 
 \node[v](v3) at (15,18) {};
 
 \node[v](v4) at (30,36) {};
 \node[v](v5) at (38,13) {};
 
 \draw[very thick, gray] (v1) -- (v2)--(v3);
 \draw (v3)--(v1);
\draw[very thick, gray] (v4) -- (v5);
 
\node at (31,25) {$f'$}; 
\node at (25,16) {$C_1$}; 
\node at (41,18) {$C_2$}; 
  \node at (15,33) {$u$}; 
  \node at (15,15) {$w$}; 
\node at (42,36) {$e$}; 
\node at (17,24) {$f$}; 
\node at (10,21) {$v$}; 

  \node at (27,9) {};  
\end{tikzpicture}\qquad
 \begin{tikzpicture}[scale=0.1]
 \draw (15,18) arc (210:360:18 and 12);
 \draw (15,30) arc (150:0:18 and 12);  
 \draw[very thick] (37,35) arc (60:30:19 and 9);  
 
 \node[v](v1) at (15,30) {}; 
  \node[v](v7) at (30,36) {}; 
 \node[v](v2) at (10,24) {}; 
 \node[v](v3) at (15,18) {};
 
 \node[v](v4) at (15,30) {};
 \node[v](v5) at (27,12) {};
 \node[v](v6) at (20,14) {};
 \draw (v1) -- (v7) -- (v5);
 \draw (v1) -- (v2)--(v3)--(v1);
 \draw (v4) -- (v5);
 \draw (v3) -- (v5);
\node at (21,25) {$f'$}; 
  \node at (15,33) {$u$}; 
  \node at (15,15) {$w$}; 
  \node at (20,11) {$x$}; 
  \node at (27,9) {$y$}; 
\node at (42,36) {$e$}; 
\node at (14,25) {$f$}; 
\node at (10,21) {$v$}; 
\node at (30,39) {$z$};

\end{tikzpicture}
\caption{The cycle $C$ of length at least nine in Proposition~\ref{prop:contractionsc} and two even cycles $C_1$ and $C_2$. The second picture depicts the last case that $C_1$ has four edges and $f$ and $f'$ meet at $u$.}\label{fig:goodchord}
\end{figure}

	Now suppose that $C$ has length at least nine and the assertion holds for all odd cycles shorter than $C$.
	Since $G$ is chordal, $C$ has a chord, say $f=uw$, connecting
	two vertices at distance two on $C$. 
	See Figure~\ref{fig:goodchord}.
	Let $v$ be the common neighbor of $u$ and $w$ in $C$, 
	and let $C'$ be the cycle $(C- v) + f$.
	Since $C'$ is an even cycle of length at least eight  in $G$, 
	it has an odd chord of $C'$, say $f'$.
	Let $C_1, C_2$ be the two distinct cycles in $C' + f'$ containing $f'$
	such that $C_1$ contains the edge $f$. 
	If $e\notin E(C_2)$ then
	$f'$ is an odd chord of $C/e$. 
	Thus we may assume that $e\in E(C_2)$.
	We consider two cases.
	
	\smallskip
	\noindent\emph{Case 1. $C_1$ has length at least six.}
	Here, the cycle $(C_1- f) + uv + vw$ has length at least seven in $G$. 	
	So by the induction hypothesis, 
	$((C_1- f) + uv + vw)/f'$ has an odd chord $h$ in $G/f'$. 
	This chord was a chord in $(C_1- f) + uv + vw$ where the part avoiding $f'$ has odd edges with at least three edges. 
	Therefore, $h$ is an odd chord of $C/e$.
	
	\smallskip
 	\noindent\emph{Case 2. $C_1$ has length four.}
 	See the second picture in Figure~\ref{fig:goodchord}. 
	If $f$ does not meet $f'$, 
	then any chord of the cycle $C_1$ is an odd chord in $C/e$.
	So, we may assume that 
	$C$ contains the path $u-v-w-x-y$ and $f'=uy\in E$.
	If $ux\in E$, then it is an odd chord for $C/e$. 
	We may assume that $ux\notin E$ and $wy\in E$. 

	Since $C_2$ is a cycle of $G$,
	there exists a vertex $z$ in $C_2$
	other than $u$ and $y$ such that
	$uyz$ is a triangle in $G$.
	We claim that $G[\{u,v,w,x,y,z\}]$ is isomorphic to $S_3$.
	Since $\omega(G)\le 3$, $G$ has no edges $ux$, $yv$, $wz$.
	If $G$ has one of edges $xv, vz, zx$, then 
	one of the sets $\{x,v,u,y\}$, $\{v,z,w,y\}$, $\{z,x,u,w\}$ form 
	a cycle of length four in $G$,
	and this cycle forces one of the edges $ux$, $yv$, $wz$.
	Therefore, $G$ also has no edges $xv, vz, zx$.
	So, $G[\{u,v,w,x,y,z\}]$ is isomorphic to $S_3$, contradicting to the assumption that $G$ is strongly chordal.
	
	\smallskip
	We conclude that for every even cycle of length at least six in $G/e$ has an odd chord.
	Therefore, $G/e$ is strongly chordal.
	\end{proof}
		\begin{proof}[Proof of Proposition~\ref{prop:contraction}]
	Let $G$ be a graph and 
	suppose that there exists a strongly chordal graph $H$ of $\omega(H)\le 3$ such that $G$ is a subgraph of $H$.
	Clearly, taking a subgraph does not increase strongly chordal treewidth.
	Also, for $e\in E$, $G/e$ is a subgraph of $H/e$ and
	by Proposition~\ref{prop:contractionsc}, $H/e$ is also strongly chordal.
	Therefore, $\sctw(G/e)\le 2$.
	\end{proof}

\subsection{Characterization with Cycle Model}
	
	We characterize the class of strongly chordal treewidth two in terms of a cycle model.
		
	\begin{THM}\label{thm:main2}
	Let $G=(V,E)$ be a $2$-connected graph. Then 
	$G$ has strongly chordal treewidth two if and only if
	the cell completion $\widetilde{G}$ of $G$ is a tree of two-boundaried cycles. 	
\end{THM}

	\begin{PROP}\label{prop:nominorsimple}
	Let $G=(V,E)$ be a $2$-connected graph.
	If $G$ has strongly chordal treewidth two,
    then $\widetilde{G}$ is a tree of two-boundaried cycles.	
	\end{PROP}
	
	\begin{proof}
	Suppose $G$ has strongly chordal treewidth two.
	Since $G$ is $2$-connected and has treewidth two,	
	by Theorem~\ref{thm:babette}, $\widetilde{G}$ is a tree of cycles.
	Suppose that a chordless cycle $C$ of $\widetilde{G}$ has edge separators $\{v_iw_i\}_{1\le i\le k}$ where $k\ge 3$.
	By Lemma~\ref{lem:avoidingcycle}, for each $1\le i\le k$, 
	$G$ has two internally vertex-disjoint paths $P^i_1$, $P^i_2$ from $v_i$ to $w_i$ in $G$ such that they have no vertices of $C$ except $v_i$ and $w_i$.
	So, $C- v_1w_1- v_2w_2 \cdots - v_kw_k$ with the paths $\bigcup_{1\le i\le k} \{P^i_1, P^i_2\}$ in $G$ has a minor isomorphic to $S_3$. 
	By Lemma~\ref{lem:sctws3} and Proposition~\ref{prop:contraction},
	$\sctw(G)\ge 3$ and it contradicts to the assumption on $G$.
		\end{proof}

	For the opposite direction, we prove the following. 
		
	\begin{PROP}\label{prop:simpletree2}
	A tree of two-boundaried cycles has strongly chordal tree-width two.
	\end{PROP}

	\begin{proof}
	Let $G=(V,E)$ be a tree of two-boundaried cycles. 
	We will construct a graph $G'$ from $G$ by triangulating each chordless cycle such that $G'$ is a strongly chordal graph with $\omega (G')=3$.

			\begin{figure}[t]\centering
 \tikzstyle{v}=[circle, draw, solid, fill=black, inner sep=0pt, minimum width=3pt]
\begin{tikzpicture}[scale=0.1]

 \node[v](v1) at (0,30) {}; 
 \node[v](v2) at (30,30) {}; 
 \node[v](v3) at (3,37) {}; 
 \node[v](v4) at (27,37) {}; 
 \node[v](v5) at (3,23) {}; 
 \node[v](v6) at (27,23) {}; 
  \foreach \x in {20}
 {
 \node[v] at (10,\x) {}; 
 \node[v] at (20,\x) {};
 \draw (v1) --(v5)-- (10,\x)--(20,\x)--(v6)--(v2);
 }
  \foreach \x in {40}
 {
 \node[v] at (10,\x) {}; 
 \node[v] at (20,\x) {};
 \draw (v1) --(v3)-- (10,\x)--(20,\x)--(v4)--(v2);
 }
 \draw[red, thick] (v1) --(v5);
 \draw[red, very thick] (v6) --(v2);
 
 \node at (-2,33) {$c_1$}; 
 \node at (3,40) {$c_2$}; 
 \node at (10,43) {$c_3$}; 
 \node at (20,43) {$c_4$}; 
 \node at (27,40) {$c_5$}; 
 \node at (32,33) {$c_6$}; 

 \node at (3,20) {$d_1$}; 
 \node at (10,17) {$d_2$}; 
 \node at (20,17) {$d_3$}; 
 \node at (27,20) {$d_4$}; 

 \draw (v1) -- (10,20);
 \draw (v1) -- (20,20);
 \draw (v1) -- (v6);

  \draw (v6) -- (v3);
 \draw (v6) -- (10,40);
 \draw (v6) -- (20,40);
 \draw (v6) -- (v4);
 
 \node[v](w1) at (-5,32) {}; 
 \node[v](w2) at (-3,20) {}; 
\node[v](w5) at (-10,32) {}; 
 \node[v](w6) at (-8,20) {}; 

  \node[v](w3) at (35,32) {}; 
 \node[v](w4) at (33,20) {}; 
  \node[v](w7) at (40,32) {}; 
 \node[v](w8) at (38,20) {}; 
 
 \draw (v1)--(w1)--(w5);
 \draw (v5)--(w2)--(w6);
 \draw (v2)--(w3)--(w7);
 \draw (v6)--(w4)--(w8);

 \draw (w5)--(v5)--(w1);
 \draw (w4)--(v2)--(w8);

\node at (-10,25) {$\cdots$}; 
\node at (40,25) {$\cdots$}; 
\end{tikzpicture}
\caption{Triangulating each chordless cycle in Proposition~\ref{prop:simpletree2}.}\label{fig:triang}
\end{figure}
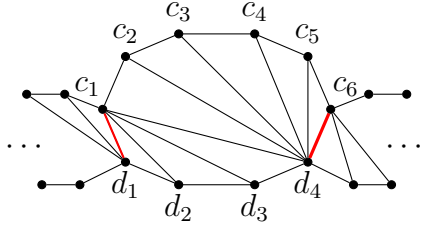

	By the definition of a tree of two-boundaried cycles, 
	each chordless cycle of $G$ has at most two edge separators.
	For convenience, we choose up to two non-edge separator edges in each chordless cycle  
	and call them also edge separators 
	so that each chordless cycle has exactly two edge separators.
	Note that for any chordless cycle $C$ of $G$ and edges $u_1v_1$ and $u_2v_2$ in $C$,
	we can triangulate $C$ into $C'$ with maximum clique size at most 
	three so that every triangle in $C'$
has an edge other than $u_1v_1, u_2v_2$, which is not contained in any other triangles of $C'$. 
	Let $G'$ be the graph obtained from $G$ by triangulating each chordless cycle as described. See Figure~\ref{fig:triang}.
			
	Since $\omega (G')=3$, it remains to prove that $G'$ is strongly chordal.
	Suppose that $G'$ is not strongly chordal.
	Since $G'$ is a chordal graph and $\omega (G')=3$, by Theorem~\ref{thm:charsc},
	$G'$ has an induced subgraph isomorphic to $S_3$.
	However, since every triangle of $G'$ has an edge 
	which is not contained in other triangles of $G'$, it leads a contradiction.
	\end{proof}

\subsection{The Minor Obstruction Set for Strongly Chordal Treewidth Two}
	We provide the minor obstruction set for the class of graphs of strongly chordal treewidth at most two.
	\begin{THM}\label{thm:sctwmo}
	Let $G$ be a $2$-connected graph. 
	The graph $G$ has strongly chordal treewidth two
	if and only if it has no minor isomorphic to $K_4$ or $S_3$.
\end{THM}
	\begin{proof}
	If $G$ has a minor isomorphic to $K_4$ or $S_3$,
	then by Lemma~\ref{lem:sctws3} and Proposition~\ref{prop:contraction},
	$\sctw(G)\ge 3$.
	Suppose $G$ has no minor isomorphic to $K_4$ and $\sctw(G)\ge 3$.
	Then, by Theorem~\ref{thm:main2}, $\widetilde{G}$ is a tree of cycles but not a tree of two-boundaried cycles. 
	Therefore, $\widetilde{G}$ has a chordless cycle having at least three edge separators.
	As we already observed in Proposition~\ref{prop:nominorsimple},
	by Lemma~\ref{lem:avoidingcycle}, we can easily show that $G$ has a minor isomorphic to $S_3$.	
	\end{proof}

	\begin{proof}[Proof of Theorem~\ref{thm:mainsc}]
	From the definition of strongly chordal treewidth,
the strongly chordal treewidth of a graph is the maximum of this parameter over all blocks of it.
	With this observation, by Theorem~\ref{thm:main2}, the statements (1) and (2) are equivalent.
	Also, Theorem~\ref{thm:sctwmo} implies that (1) and (3) are equivalent.
	\end{proof}

\section{Characterizations of Directed Spaghetti \\ Treewidth Two and Special Treewidth Two}
\label{section:special}

    In this section, we mainly characterize the class of graphs having special treewidth at most two, and the class of graphs having directed spaghetti treewidth.
    
    A graph is called a \emph{mamba}\footnote{Mambas are a type of snakes.} if 
it is either a 2-connected graph of pathwidth two, or a single edge, or an isolated vertex. 
	The notion of mambas reflects the linear structure of them, 
	and we will define head vertices of mambas which have a key role in our characterization.	
	
	We first 
    show that every block of a graph of directed spaghetti treewidth at most two, or special treewidth at most two is a mamba.
	For directed spaghetti treewidth, we directly obtain a characterization of width at most two for general cases.
	For special treewidth, we will see how the different mambas are glued to make graphs of special treewidth at most two.

\subsection{Mambas}

	We obtain the following characterization of $2$-connected mambas as a corollary of the results in Section~\ref{section:spaghetti} and Section~\ref{section:stronglychordal}.
	\begin{COR}\label{cor:main6}
	Let $G$ be a $2$-connected graph. The following are equivalent.
	\begin{enumerate}
	\item $G$ has pathwidth two (equivalently, $G$ is a mamba).
	\item $G$ has special treewidth two.
	\item $G$ has directed spaghetti treewidth two.
	\item $G$ has no minor isomorphic to $K_4$, $D_3$ or $S_3$.
	\item The cell completion $\widetilde{G}$ of $G$ is a path of cycles. 	
	\end{enumerate}
	\end{COR}

   For the direction $(3)\Rightarrow (4)$, we prove that every DV graph with clique number at most three is strongly chordal. This gives a relation between directed spaghetti treewidth and strongly chordal treewidth when the parameter is at most two. 
   A sun is called \emph{even} (or \emph{odd}) if the size of the central clique is even (or odd).
   \begin{THM}[Panda~\cite{Panda99}]\label{thm:panda}
   Every DV graph has no induced subgraph isomorphic to an odd sun.
   \end{THM}
  
	\begin{LEM}\label{lem:dpmax3}
	Every DV graph with clique number at most three
	is strongly chordal. 
	\end{LEM}
	\begin{proof}
	Let $G$ be a DV graph with clique number three.
	Since $G$ is a DV graph, 
	by Theorem~\ref{thm:panda}, $G$ has no induced subgraph isomorphic to an odd sun.
	Since $\omega(G)\le 3$, 
	$G$ has no induced subgraph isomorphic to an even sun.
	Therefore, 
	by Theorem~\ref{thm:charsc},
	$G$ is strongly chordal.
	\end{proof}

	\begin{proof}[Proof of Corollary~\ref{cor:main6}]
	By Theorem~\ref{thm:babette2}, $(5)$ implies $(1)$, and from the inequalities between the parameters, $(1)$ implies $(2)$ and $(2)$ implies $(3)$.
	
	$(3)\Rightarrow (4)$ :
	If $G$ has directed spaghetti treewidth two,
	then $G$ has spaghetti treewidth at most two.
	Also, by Lemma~\ref{lem:dpmax3}, 
	if $G$ has directed spaghetti treewidth two, then $G$ has strongly chordal treewidth at most two.
	Since $G$ is $2$-connected, 
	by Theorem~\ref{thm:sptwmo} and~\ref{thm:sctwmo},
	$G$ has no minor isomorphic to $K_4$, $D_3$ or $S_3$.	
	
	$(4)\Rightarrow (5)$ :
	Suppose $G$ has no minor isomorphic to $K_4$, $D_3$ or $S_3$.
	Since $G$ is $2$-connected, 
	by Proposition~\ref{prop:main3} and \ref{prop:nominorsimple},
	$\widetilde{G}$ is both a chain tree of cycles and a tree of two-boundaried cycles.
	By the definition of a path of cycles, $\widetilde{G}$ is a path of cycles.
\end{proof}

The graphs of directed
spaghetti treewidth at most two are exactly the graphs whose block is a mamba.

\begin{THM}\label{theorem:dstwblocks}
Let $G=(V,E)$ be a graph. The following are equivalent.
\begin{enumerate}
\item $G$ has directed spaghetti treewidth at most two.
\item Each block of $G$ is a mamba.
\item $G$ has no minor isomorphic to $K_4$, $D_3$ or $S_3$.
\end{enumerate}
\end{THM}
\begin{proof}
Similarly in the proof of Lemma~\ref{lem:sptwblock}, we can easily verify that $G$ has directed spaghetti treewidth at most two if and only if every block of it has directed spaghetti treewidth at most two.
So, all directions are easily obtained from Corollary~\ref{cor:main6}.
\end{proof}

\subsection{Characterizing Graphs of Special Treewidth Two}
\label{section:charsp2}
A result like Theorem~\ref{theorem:dstwblocks} does not hold for special treewidth 
two: we can have a graph with special treewidth at least three, where each block
has special treewidth at most two. For instance, the graphs $G_1$, $G_2$,
and $G_3$ in Figure~\ref{figure:forbidden} have special treewidth three 
but one can easily observe that
each block is a mamba.
Thus, for a graph to have special treewidth at most two, it is necessary but
not sufficient that each block is a mamba. 
An additional condition, expressing how the different mambas are attached to each other, is given below; adding this condition gives
a full characterization.

Head vertices
of mambas play a central role in the characterizations. 

\begin{definition}\label{def:headvertex}
Let $G=(V,E)$ be a mamba. A vertex $v\in V$ is a {\em head vertex} of $G$,
if there is a path decomposition $(T,\{B_x\}_{x\in V(T)})$ of $G$ having width at most two such that $T=p_1p_2 \cdots p_r$ and $v\in B_{p_1}$.
\end{definition}

We define \emph{mamba trees}, which are recursively constructed by attaching mambas at head vertices. 
We will show that mamba trees precisely characterize
the connected graphs of special treewidth at most two.
In the next section,
we characterize this class using forbidden minors.

\begin{definition}
\label{def:mambatree2}
The class of \emph{mamba trees} is the class of graphs recursively defined as follows.
\begin{itemize}
	\item Each mamba is a mamba tree.
	\item For each mamba tree $G$ and each mamba $M$, the graph obtained from a disjoint union of $G$ and $M$ by identifying a vertex of $G$ with a head vertex of $M$ is a mamba tree.
\end{itemize}
\end{definition}

\begin{THM}
\label{theorem:spctw2}
A graph has special treewidth at most two,
if and only if it is a disjoint union of mamba trees.
\end{THM}

Let $\mathcal{I}$ be a special tree decomposition with a root bag $R$.
A bag $B_y$ is a \emph{descendant} of $B_x$ in $\mathcal{I}$, if $B_x$ belongs to
the path of  $\mathcal{I}$ linking $B_y$ to the root $R$. If $B_y$ is a descendant of $B_x$ and
they are neighbor bags, then $B_x$ is called a \emph{parent} of $B_y$, and $B_y$ is a \emph{child} of $B_x$.

\begin{PROP}
\label{prop:mambatree}
The special treewidth of a mamba tree is at most two.
\end{PROP}
\begin{proof}
Let $G=(V,E)$ be a mamba tree. If $G$ is a mamba, 
then by Corollary~\ref{cor:main6}, 
$G$ has special treewidth at most two.

Otherwise, we know from Definition~\ref{def:mambatree2} that $G$ is constructed by taking a mamba tree $G'$ and a mamba $M$, and identifying a vertex in $G'$ with a head vertex in $M$. Let $v$ be this vertex.
From Definition~\ref{def:headvertex} we know that since $v$ is a head vertex, there is a path decomposition $\mathcal{I}^v=(T^v, \{B^v_x\}_{x\in V(T^v)})$ of $M$ of width at most two such that $v$ is in the first bag $B_1$.
Now iteratively 
we get a special tree decomposition $\mathcal{I}$ of $G$ of width at most two from 
a special tree decomposition $\mathcal{I}'=(T', \{B'_x\}_{x\in V(T')})$ of $G'$ by attaching $T^v$ to $T'$ and making $B_1$ the child of the lowest bag of $P(\mathcal{I}', v)$ in $\mathcal{I}'$. So, we conclude that every mamba tree has special treewidth at most two.
\end{proof}

\begin{PROP}\label{prop:spctw2mamba}
Every connected graph having special treewidth at most two is a mamba tree.
\end{PROP}

\begin{proof}
Let $G=(V,E)$ be a connected graph of special treewidth at most two.
We prove by induction on the number of vertices in $G$.
If $\abs{V}\le 3$, then this is always true.
We may assume that $\abs{V}\ge 4$.

We choose a special tree decomposition $\mathcal{I}=(T, \{B_x\}_{x\in V(T)})$ of $G$ having width at most two such that $\sum_{x\in V(T)} \abs{B_x}$ is minimum. Note that $T$ is a rooted tree and
if $B$ is a bag of $\mathcal{I}$, then 
$B$ has at most one child $B'$ such that $\abs{B\cap B'}\ge 2$, otherwise, there must be a vertex of $G$ where the bags containing it do not form a rooted path.

We choose a maximal rooted path $P=B_1-B_2-\cdots-B_n$ in $T$ such that 
for all $1\le i\le n-1$, $\abs{B_i\cap B_{i+1}}=2$ and $B_i$ is a child of $B_{i+1}$ in $T$.
We show that $G[\bigcup_{1\le i\le n}B_i]$ is $2$-connected, and so, it is a mamba.
To show this, we analyze some cases forcing edges in the graph.

\begin{claim}\label{claim:smaller}
Let $t\in V(T)$ and let $B'$ be the parent of $B_t$ such that $B_t\cap B'=\{v_1, v_2\}$ and $B_t- B'=\{w\}$. If there is no child $B''$ of $B_t$ such that $\abs{B_t\cap B''}=2$, then $w$ is adjacent to both $v_1$ and $v_2$.
\end{claim}
\begin{proof}
Suppose $v_i$ is not adjacent to $w$ in $G$ for some $i\in \{1,2\}$ and let $v_j$ be the vertex in $B_t\cap B'$ other than $v_i$.
If $B_t$ has no child containing $v_i$, we can simply remove $v_i$ from $B_t$.
In the below of $B_t$ in $T$, if there exists a component $T'$ of $T- t$ containing a bag with $v_i$, 
then we cut the branch $T'$ from $T$, and attach this on $B'$, and remove $v_i$ from $B_t$.
Since $T'$ has no bag containing $w$ or $v_j$, the modified decomposition is a special tree decomposition and $\sum_{x\in V(T)} \abs{B_x}$ is decreased by one, contradiction.
\end{proof}

\begin{claim}\label{claim:smaller2}
Let $B_t$ be a bag of $\mathcal{I}$ and let $B'$ be a child of $B_t$ such that $B_t\cap B'=\{v_1, v_2\}$ and $B_t- B'=\{w\}$.
If $B_t$ is a non-root bag with the parent $B''$ such that $\abs{B_t\cap B''}=1$, then $w$ is adjacent to both $v_1$ and $v_2$.
\end{claim}
\begin{proof}
Suppose $v_i$ is not adjacent to $w$ in $G$ for some $i\in \{1,2\}$ and let $v_j$ be the vertex in $B_t\cap B'$ other than $v_i$.
 If $B_t\cap B''\neq \{v_i\}$, then we can remove $v_i$ from the bag $B_t$.
Thus, we may assume that $B_t\cap B''= \{v_i\}$.
Let $L$ be the bag of $P(\mathcal{I}, v_j)$ where the distance from $L$ to the root is maximum.
If $B_t$ has a child $B_z$ containing $w$, let $T_z$ be the subtree of $T-t$ containing $z$.

Let $\mathcal{I'}=(T', \{B_x\}_{x\in V(T')})$ be a decomposition obtained by removing $B_t$, and connecting $B'$ and $B''$, and adding a new bag $\{w, v_j\}$ on the bag $L$, and attaching $T_z$ on the new bag so that $B_z$ is a child of the bag $\{w, v_j\}$, if exists. 
The resulting decomposition is a special tree decomposition and $\sum_{x\in V(T)} \abs{B_x}$ is decreased by one, which leads a contradiction.
\end{proof}

\begin{claim}\label{claim:smaller3}
Let $B_t$ be a non-root bag of $\mathcal{I}$ with a child $B'$ and the parent $B''$ such that 
$B_t\cap B'=B_t\cap B''=\{v_1, v_2\}$. 
If  $B_t- B'=\{w\}$,
then $w$ is adjacent to both $v_1$ and $v_2$.
\end{claim}
\begin{proof}
Suppose $v_i$ is not adjacent to $w$ for some $i\in \{1,2\}$ and let $v_j$ be the vertex of $\{v_1, v_2\}$ other than $v_i$.
Similarly in Claim~\ref{claim:smaller2},
we first remove the bag $B_t$, and connect $B'$ and $B''$,
and add a new bag $\{w, v_j\}$ on the below of the path $P(\mathcal{I}, v_j)$, and  
if there is a component of $T-t$ containing a bag with the vertex $w$, then cut and attach it on the new bag $\{w,v_j\}$.
The resulting decomposition is a special tree decomposition and $\sum_{x\in V(T)} \abs{B_x}$ is decreased by one, which leads a contradiction.
\end{proof}

\begin{claim}\label{claim:smaller4}
Let $B$ be a non-root bag of $\mathcal{I}$ with a child $B'$ and the parent $B''$ such that 
$B\cap B'=\{w, v_1\}$ and $B\cap B''=\{w, v_2\}$. 
Then $v_1$ is adjacent to $v_2$.
\end{claim}
\begin{proof}
Suppose $v_1$ is not adjacent to $v_2$. 
Note that $v_1$ is contained in neither $B''$ nor any child of $B$ other than $B'$. So, we can remove $v_1$ from the bag $B$ and reduce $\sum_{x\in V(T)} \abs{B_x}$.
It contradicts to the minimality of $\sum_{x\in V(T)} \abs{B_x}$.
\end{proof}

Now we show that $G[\bigcup_{1\le i\le n}B_i]$ is $2$-connected.
Since $\sum_{x\in V(T)} \abs{B_x}$ is minimum,
$B_1-B_2\neq\emptyset$ and $B_n-B_{n-1}\neq\emptyset$.
If $B_1- B_2=\{x\}$,
then by Claim~\ref{claim:smaller}, $x$ is adjacent to both vertices in $B_1\cap B_2$.
If $B_{n}- B_{n-1}=\{x\}$,
then by Claim~\ref{claim:smaller2}, $x$ is adjacent to both vertices in $B_{n-1}\cap B_n$.

Suppose $B_{i-1}$, $B_i$, $B_{i+1}$ are three consecutive bags in $P$. We have two cases.
If $B_{i-1}\cap B_i=B_i\cap B_{i+1}$, 
then by Claim~\ref{claim:smaller3}, the vertex $w$ in $B_i- B_{i-1}$ is adjacent to both vertices in $B_{i-1}\cap B_i$.
Suppose $B_{i-1}\cap B_i\neq B_i\cap B_{i+1}$.
Let $B_{i-1}- B_i =\{y\}$ and $B_{i+1}- B_i =\{z\}$. 
In this case, by Claim~\ref{claim:smaller4}, $y$ is adjacent to $z$ in $G$.
From these analysis, it is easy to verify that 
$G[\bigcup_{1\le i\le n}B_i]$ is $2$-connected, and therefore
$G[\bigcup_{1\le i\le n}B_i]$ is a mamba.

Now we show that $G$ is a mamba tree.

We may assume that there exists a non-root bag $B$ of $\mathcal{I}$ having the parent $B'$ such that $\abs{B\cap B'}=1$, otherwise $G$ is a mamba.
We choose such a bag $B$ so that the distance from $B$ to the root is maximum and assume that $B\cap B'=\{v\}$.
Let $P$ be the union of $B$ and all descendants of $B$.
If $\abs{P}\le 3$, then $G[P]$ consists of either one or two blocks of size two, or a triangle. 
If $\abs{P}\ge 4$, then
$G[P]$ is a mamba and since $v\in B$, $v$ is a head vertex of this mamba.
By the induction hypothesis, $G[(V- P) \cup \{v\}]$ is a mamba tree. In all cases, we conclude that $G$ is a mamba tree.
\end{proof}

\begin{proof}[Proof of Theorem~\ref{theorem:spctw2}]
As the special treewidth of a graph is the maximum of the special 
treewidth of its connected components, by Proposition~\ref{prop:mambatree} it
follows that any disjoint union of mamba trees has special treewidth at most two.
If a graph has special treewidth at most two, then
by Proposition~\ref{prop:spctw2mamba}, it is a disjoint union of mamba trees. 
\end{proof}

\subsection{The Minor Obstruction Set for Special Treewidth Two}
\label{section:obstructionspecial}

This section is devoted to the proof of Theorem~\ref{theorem:obstruction}, given below.

\begin{THM}\label{theorem:obstruction}
A graph has special treewidth at most two if and only if it has no minor isomorphic to $K_4, S_3, D_3, G_1, G_2$, or $G_3$.
\end{THM}

In Figure~\ref{figure:forbidden}, the graphs $G_1, G_2, G_3$ in the obstruction set are displayed.

From our structural characterization of graphs of special treewidth at most two of the previous sections,
we can easily check that the class is minor closed.

Also, Proposition~\ref{prop:spctw2mamba} immediately follows that every graph of $\{G_1, G_2, G_3\}$ has special treewidth at least three, and
a tedious case analysis shows that each proper minor of a graph in $\{K_4, S_3, D_3, G_1,G_2,G_3\}$ has special treewidth at most two.
So $\{K_4, S_3, D_3, G_1,G_2,G_3\}$ is a subset of the obstruction set for the class of graphs of special treewidth at most two.
Thus, Theorem~\ref{theorem:obstruction} follows from the next lemma.

\begin{LEM}\label{lem:minorcharsctw2}
If a graph $G=(V,E)$ contains no minor isomorphic to $K_4$, $S_3$, $D_3$, $G_1$, $G_2$ or $G_3$, then
the special treewidth of $G$ is at most two.
\end{LEM}

To show Lemma~\ref{lem:minorcharsctw2}, we extend the standard minor notion to pairs of a graph and a vertex. For pairs $(G, v)$ and $(H,v)$, with $G =
(V,E)$, $H = (W, F)$, $v\in V$, $v\in W$, 
we say that $(H,v)$ is a \emph{minor} of $(G, v)$, if
we can obtain $H$ from $G$ by a series of the following operations: deletion of a vertex other than $v$, deletion of an edge, and contraction of an edge, such that
whenever we contract an edge with $v$ as an endpoint, the contracted vertex is named $v$. 
For pairs $(G, v)$ and $(H,w)$, we say $(G, v)$ and $(H,w)$ are \emph{isomorphic} if there is a graph isomorphism $f$ from $G$ to $H$ with $f(v) = w$.

\begin{figure}[t]
\begin{center}
\includegraphics[scale=0.25]{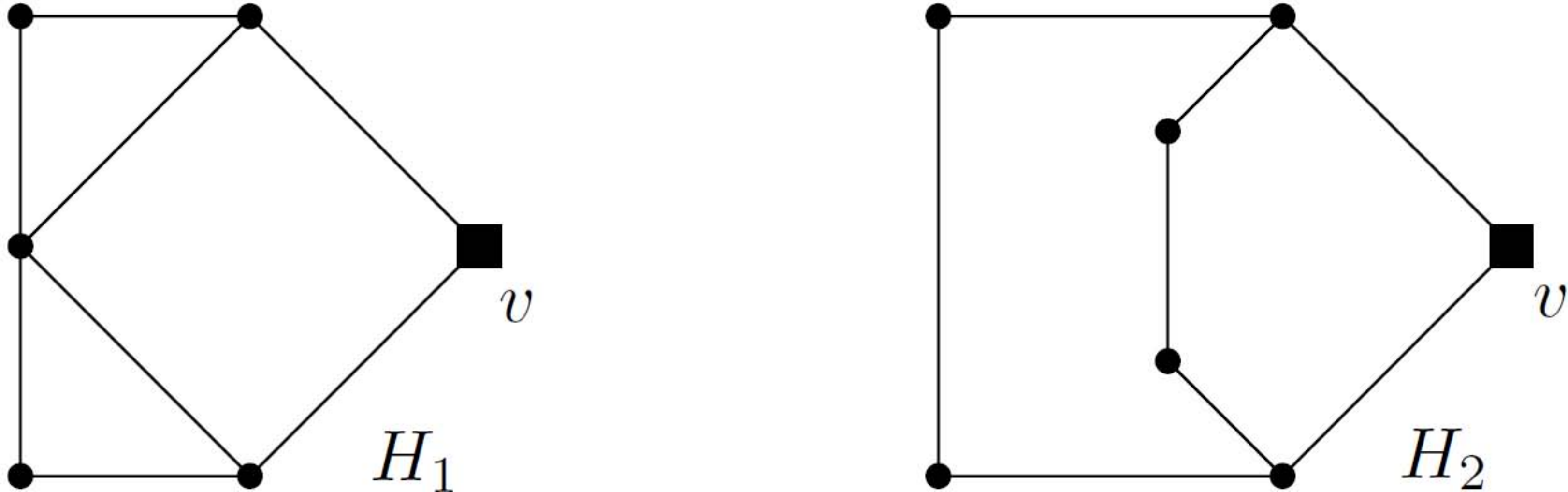}
\end{center}
\caption{Blocks of obstructions that are not 2-connected}
\label{figure:component}
\end{figure}

The following lemma is a key lemma to find a minor isomorphic to $G_1, G_2$ or $G_3$. The pairs $(H_1,v)$ and $(H_2, v)$ are depicted in Figure~\ref{figure:component}, with $v$ the marked vertex.

\begin{LEM}\label{lem:specialblock}
Let $B$ be a $2$-connected mamba and let $z$ be a vertex which is not a head vertex of $B$.
Then $(B,z)$ has a minor isomorphic to either $(H_1,v)$ or $(H_2, v)$.
\end{LEM}
\begin{proof}
Since $B$ is a $2$-connected mamba, by Theorem~\ref{thm:babette2}, $\widetilde{B}$ is a path of cycles. 
Let $(U, F)$ be a cycle path model of $\widetilde{B}$  with $U= (C_1,\ldots,C_p)$ and $F = (f_1,\ldots,f_{p-1})$. 
We may assume that 
\begin{enumerate}
\item $C_1=C$ if $\widetilde{B}$ has a chordless cycle $C$ of length at least four containing exactly one edge separator $f_1$, and
\item $C_p=C'$ if $\widetilde{B}$ has a chordless cycle $C'$ of length at least four containing exactly one edge separator $f_{p-1}$.
\end{enumerate}
By removing repeated edges from $F$, we can obtain a linear ordering $e_1, \ldots, e_r$ of all edge separators of $\widetilde{B}$. 
If $r=1$, then we can easily observe that either $(B,z)$ has a minor isomorphic to $(H_2, v)$, or there is a path decomposition of $\widetilde{B}$ having width two such that the first bag contains $z$. Therefore, we may assume that $r\ge 2$.

For each $i\in \{1,r\}$, let $\mathcal{C}_i$ be the set of all chordless cycles of $\widetilde{B}$ containing exactly one edge separator $e_i$.
We observe the following.
\begin{enumerate}
\item If $C_i$ is not a simplicial triangle, then $z\notin V(C_i)$ but $z$ can be a vertex of degree two in a simplicial triangle in $\mathcal{C}_i$.
\item If $C_i$ is a simplicial triangle, then $z\notin \bigcup_{C\in \mathcal{C}_i} V(C)$.
\end{enumerate}

We fix $i\in \{1,r\}$ and $e_i=u_iv_i$. 
If $C_i$ is a simplicial triangle, then all cycles in $\mathcal{C}_i$ are simplicial triangles. So, the second statement is true because 
if  $z\in \bigcup_{C\in \mathcal{C}_i} V(C)$, then 
it is not hard to construct a path decomposition of $\widetilde{B}$ where the first bag contains $z$.  
Also, in the first statement, if $z\in V(C_i)$, we can easily construct a path decomposition where the first bag contains $z$.

Suppose $C_i$ is not a simplicial triangle and $z$ is a vertex of degree two in a simplicial triangle in $\mathcal{C}_i$.
In this case, since $\widetilde{B}[V(B)-\{u_i,v_i\}]$ has two components of size at least two not containing $v$,
using Lemma~\ref{lem:twocompos},
$B$ has two internaly vertex-disjoint paths of length at least three from $u_i$ to $v_i$. 
Then $(B,z)$ has a minor isomorphic to $(H_2, v)$. 
Since there is no path decomposition of $H_2$ having width two where the first bag contains $v$, we conclude that $z$ is not a head vertex.  

Now we show that for each case, $(B,z)$ has a minor isomorphic to either $(H_1,v)$ or $(H_2, v)$.
We may assume that there exist two vertex-disjoint paths $P_1, P_2$ in $\widetilde{B}$ where $P_1$ links $u_1$ to $u_r$ and $P_2$ links $v_1$ to $v_r$.

We have two cases.

\smallskip
\noindent\emph{Case 1. $z$ is the vertex of degree two in a simplicial triangle of $\widetilde{B}$ having an edge separator $xy$.}
If $xy=e_1$ (or $xy=e_r$), then from the above observation, $C_1$ (or $C_r$) is a cycle of length at least four. 
So, in any cases, we have that 
$\widetilde{B}[V(B)-\{x,y\}]$ has two components of size at least two, which have the vertices of $C_1$ and $C_p$, respectively. 
By Lemma~\ref{lem:twocompos}, 
$B$ has two internally vertex-disjoint paths of length at least three from $x$ to $y$. Therefore, $(B, z)$ has a minor isomorphic to $(H_2, v)$.

\smallskip
\noindent\emph{Case 2. $z\in (V(P_1)\cup V(P_2))\setminus \{u_1,  u_r, v_1, v_r\}$.}
If we use a chordless cycle which contains $e_1$ and $e_2$, then by Lemma~\ref{lem:avoidingcycle}, $B$  has two internally vertex-disjoint paths from $u_1$ to $v_1$ such that they have no vertices of $(V(P_1)\cup V(P_2))\setminus \{u_1, v_1, u_r, v_r\}$. By the same reason, $B$ has two internally vertex-disjoint paths from $u_r$ to $v_r$ such that they have no vertices of $(V(P_1)\cup V(P_2))\setminus \{u_1, v_1, u_r, v_r\}$.

By symmetry, we may assume that $z\in V(P_1)\setminus \{u_1, u_r\}$. On the path $P_1$, the distance between $z$ and $u_1$ (or $u_r$) is at least one.
So, by contracting all edges of $P_2$, we get a minor isomorphic to $(H_1, v)$ together with the paths which we obtained before.
\end{proof}

\begin{proof}[Proof of Lemma~\ref{lem:minorcharsctw2}]
Suppose that the lemma does not hold. Let $G$ be a minimal counterexample such that no minor of $G$ is a counterexample.
Since the special treewidth of a graph is the maximum of the special 
treewidth of its connected components,
we may assume that $G$ is connected.
Since $G$ has no minor isomorphic to $K_4, S_3$, or $D_3$,  by Corollary~\ref{cor:main6}, each block of $G$ is a mamba.
We may also assume that $G$ has at least two blocks.

We use the well known fact that the blocks of a connected graph form a tree, called the block tree. We choose a block $B$ of $G$ corresponding to a leaf of the block tree, and let $B'$ be the block having an intersection $v$ with $B$.
If $v$ is a head vertex of $B$, 
then from the minimality of $G$, $G[(V-B)\cup\{v\}]$ is a mamba tree, 
and $G$ is again a mamba tree. 
So, we may assume that every block of $G$ corresponding to a leaf of the block tree is not attached to the remaining graph with a head vertex.

Since $G$ has at least two blocks,  $G$ has at least two blocks $B_1$ and $B_2$ corresponding to leaves of the block tree,  
	with cut vertices $z_1$ and $z_2$, respectively.
	By Lemma~\ref{lem:specialblock}, 
	each $(B_i, z_i)$ has a minor isomorphic to either $(H_1, v)$ or $(H_2, v)$. 
	Since there exists a path from $z_1$ to $z_2$ in $G$, 
	it implies that $G$ has a minor isomorphic to either $G_1, G_2$, or $G_3$, which is contradiction.
 \end{proof}

\section{Classes of Graphs having Width at most $k$ where $k\ge 3$}
\label{section:width3}
In this section, we show that for each $k\geq 3$, none of the classes of
graphs with special treewidth, spaghetti treewidth, directed spaghetti treewidth
and strongly chordal treewidth at most $k$ is closed under taking minors.

\begin{PROP}\label{proposition:closed3}
Let $k\geq 3$. Each of the following classes of graphs is not closed under taking
minors.
\begin{enumerate}
\item The graphs of special treewidth at most $k$.
\item The graphs of spaghetti treewidth at most $k$.
\item The graphs of directed spaghetti treewidth at most $k$.
\end{enumerate}
\end{PROP}
\begin{proof}
Note that trees can have arbitrary large pathwidth~\cite{EllisST94}.
Let $T=(V,E)$ be a tree with pathwidth $k$. Let $G_T=(V',E')$ be a graph obtained by taking two copies of $T$ and adding
edges between copies of vertices, that is, $V'=V_1 \cup V_2$ with $V_i = \{v_i : v\in V\}$, for $i\in \{1,2\}$,
$E = \{\{v_i,w_i\} : \{v,w\}\in E, i\in \{1,2\}\} \cup \{\{v_1,v_2\} : v\in V\}$. See Figure~\ref{figure:width3} for an example. 

\begin{figure}[t]
\begin{center}
\includegraphics[scale=1.0]{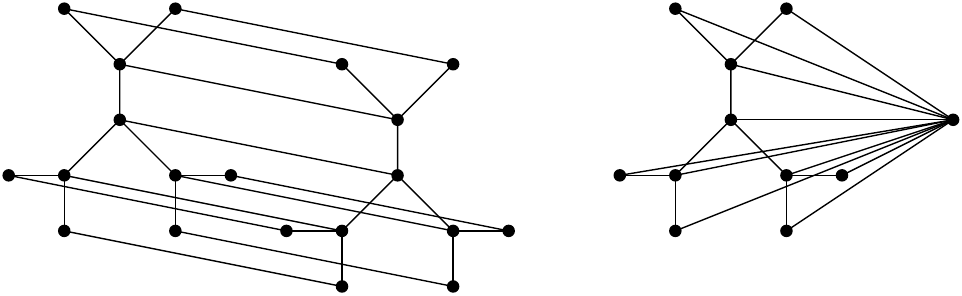}
\end{center}
\caption{$G_T$ and $G'_T$: an example of the construction in the proof of Proposition~\ref{proposition:closed3}.}
\label{figure:width3}
\end{figure}

The special treewidth of $G_T$ is at most three. With induction to the size of $T$, we show that $G_T$ has a special
tree decomposition of width at most three such that for each $v\in V$, all bags that contain $v_1$ also contain $v_2$ and
vice versa. This clearly holds when $T$ consists of a single vertex. 
Let $x\in V$ be a leaf of $T$, with the parent $y$. 
By induction, we assume we have a special tree decomposition of width at most three of 
$G_{T-x}$, such that for each $v\in V-\{x\}$, all bags that contain $v_1$ also contain $v_2$ and
vice versa. Let $i$ be a bag of maximal depth in the tree decomposition with $\{y_1,y_2\}\subseteq X_i$.
By assumption, no descendant of $i$ contains $y_1$ or $y_2$. Now add a new bag $j$ to the tree
decomposition with $i$ the parent of $j$ and $X_j = \{y_1, y_2, x_1, x_2\}$. This is a special tree decomposition
of $G_T$ of width three and  for each $v\in V$, all bags that contain $v_1$ also contain $v_2$ and
vice versa. 

Now, consider the graph $G'_T$ obtained from $G_T$ by contracting all vertices in $\{v_2 : v\in V\}$
to a single vertex $w$. 
Clearly, $w$ is adjacent to all vertices of $V(G'_T)-\{w\}$ in $G'_T$.
Hence, by Lemma~\ref{lemma:universal}, the special treewidth, spaghetti
treewidth, and directed spaghetti treewidth
of $G'_T$ equal one plus the pathwidth of $G$. So,
$G'_T$ is a minor of $G_T$ and has special treewidth, spaghetti treewidth, and
directed spaghetti treewidth exactly $k+1$.
\end{proof}

Now, we prove that the graphs of strongly chordal treewidth at most $k$ are not closed under taking minors.

	\begin{PROP}\label{prop:counterexsc}
	Let $k\ge 3$. 
	The class of graphs of strongly chordal treewidth at most $k$ is not closed under taking minors.
	\end{PROP}

	For each $k\ge 4$, we define $SC_k$ as follows. 
	For each $1\le i\le 3$, let $K_k^i$ be the complete graph on the vertex set $\{v^i_1, v^i_2, \ldots, v^i_k\}$. 
	The graph $SC_k$ is obtained from the disjoint union of $K^1_k$, $K^2_k$, $K^3_k$ and the complete graph on the vertex set $\{w_1,w_2,w_3,w_4\}$ by 
	\begin{enumerate}[label={--}] 
	\item identifying $v^1_1v^1_2$ with $w_1w_2$ (with identifying each end vertex), and $v^2_1v^2_2$ with $w_2w_3$, and $v^3_1v^3_2$ with $w_3w_4$,
	\item adding edges $w_3v^1_i$ for all $3\le i\le n-1$,
	\item adding edges $w_1v^2_i$ for all $3\le i\le n-1$,
	\item adding edges $w_2v^3_i$ for all $3\le i\le n-1$.
	\end{enumerate}
	See Figure~\ref{fig:sc5}.
	Note that $\omega(SC_{k})=k$.

	\begin{figure}[t]
 \tikzstyle{v}=[circle, draw, solid, fill=black, inner sep=0pt, minimum width=3pt]
  \centering
  \begin{tikzpicture}[scale=0.7]
	
	\draw (0.7,-2.5) node {$w2$};
	\draw (0.7,-6) node {$w1$};
	\draw (4.3,-2.5) node {$w3$};
	\draw (4.3,-6) node {$w4$};
	  \node[v](w2) at (1,-3){};
     \node[v](w1) at (1,-5.5){};
	 \node[v](w3) at (5-1,-3){};
  \node[v](w4) at (5-1,-5.5){};
    
          \draw (w1)--(w2)--(w3)--(w4)--(w1);
          \draw (w1)--(w3);
          \draw (w2)--(w4);

    \node[v](a1) at (0-1,-4){};
    \node[v](a2) at (0.5,-3.8){};
    \node[v](a3) at (0.5,-4.5){};
   
    \node[v](b2) at (1+1,-2){};
    \node[v](b3) at (5-1-1,-2){};
    \node[v](b1) at (2.5,-0.5){};
    
     \node[v](c1) at (6-0,-4){};
    \node[v](c2) at (5,-3.8){};
    \node[v](c3) at (5,-4.5){};

 \foreach \i in {1,2,3} {
          \draw (w1)--(a\i);
          \draw (w2)--(a\i);
          }
      \draw (a1)--(a2)--(a3)--(a1);
      
 \foreach \i in {1,2,3} {
          \draw (w2)--(b\i);
          \draw (w3)--(b\i);
          }
      \draw (b1)--(b2)--(b3)--(b1);

 \foreach \i in {1,2,3} {
          \draw (w3)--(c\i);
          \draw (w4)--(c\i);
          }
      \draw (c1)--(c2)--(c3)--(c1);

\foreach \i in {2,3} {
         \draw (w3) -- (a\i);
         \draw (w1) -- (b\i);
         \draw (w2) -- (c\i);}

    \end{tikzpicture}
  \caption{The graph $SC_5$.}
  \label{fig:sc5}
\end{figure}
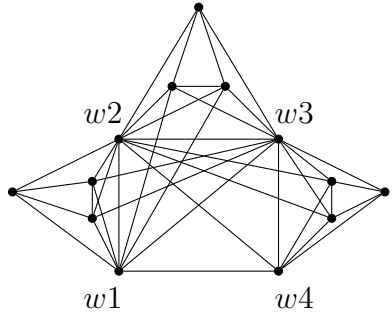

	We first show that $SC_{k+1}$ is strongly chordal.
	Let $T$ be a tree. 
	For a vertex $c\in V(T)$ and $r\ge 0$,
	we define $T(c,r)$ as the subtree of $T$ which induces on the vertices $v$ such that the distance between $v$ and $c$ is at most $r$.
	For a positive integer $k$, a graph $G=(V,E)$ is called a \emph{neighborhood subtree tolerance graph with tolerance $k$} 
	if there exists a tree $T$, and a set $S= \{ T(c_v, r_v): v\in V\}$ of subtrees of
$T$ such that 
	$xy\in E$ if and only if $\abs{V(T(c_x, r_x))\cap V(T(c_y, r_y))}\ge k$.

	Bibelnieks and Dearing showed the following.
	
	\begin{THM}[Bibelnieks and Dearing~\cite{EP1993}]\label{thm:neigh}
	Let $k$ be a positive integer.
	If $G$ is a neighborhood subtree tolerance graph with tolerance $k$, then $G$ is strongly chordal.
	\end{THM}

	\begin{LEM}\label{lem:counterex3}
	Let $k\ge 3$. The graph $SC_{k+1}$ is a neighborhood subtree tolerance graph with tolerance $1$.
	Hence $SC_{k+1}$ is strongly chordal.
	\end{LEM}
	\begin{proof}
	Let $A=a_1a_2 \cdots a_6$, $B=b_1b_2 \cdots b_5$ and $C=c_1c_2 \cdots c_7$ be paths.
	Let $T$ be the tree obtained from the disjoint union of $A, B, C$ and a new vertex $v$
	by adding edges $va_1$, $vb_1$ and $vc_1$.
We define 
	\begin{enumerate}
	\item $T_{v^1_{k+1}}=T(a_6,0)$, $T_{v^2_{k+1}}=T(b_5,0)$ and $T_{v^3_{k+1}}=T(c_7,0)$,
	\item for each $3\le j\le k$, $T_{v^1_j}=T(a_5,2)$, $T_{v^2_j}=T(b_4,2)$ and $v^3_j=T(c_6,2)$,
	\item $T_{w_1}=T(a_2,4)$, $T_{w_2}=T(a_1,6)$, $T_{w_3}=T(c_1,6)$ and $T_{w_4}=T(c_4,3)$.
	\end{enumerate}
	We can easily check that the set of subtrees $\{T_v\}_{v\in V(SC_{k+1})}$ on the tree $T$
	indeed forms a neighborhood subtree models with tolerance $1$.
	Therefore, by Theorem~\ref{thm:neigh}, $SC_{k+1}$ is strongly chordal.	
	\end{proof}

	\begin{LEM}\label{lem:counterex2}
	Let $k\ge 3$. The graph $SC_{k+1}/w_1w_4$ has strongly chordal treewidth at least $k+1$.
	\end{LEM}
	\begin{proof}
	We say $w_1$ for the contracted vertex of $SC_{k+1}/w_1w_4$. 
	Then clearly, 
	$w_1v^1_{k+1}w_2v^2_{k+1}w_3v^3_{k+1}w_1$ 
	is a cycle of length six in $SC_{k+1}/w_1w_4$ and 
	it does not have an odd chord.
	So we should add an odd chord,
	to make it as a subgraph of a strongly chordal graph.  
	We can verify that as soon as we add $v^1_{k+1}w_3$, 
	$\{w_1,w_2,w_3, v^1_3,v^1_{4}, \ldots, v^1_{k+1}\}$ becomes a clique of size $k+2$ in $SC_{k+1}/w_1w_4$.
	The same result appears when adding $v^2_{k+1}w_1$ or $v^3_{k+1}w_2$.
	Therefore, $SC_{k+1}/w_1w_4$ has a strongly chordal treewidth at least $k+1$.	
	\end{proof}

	\begin{proof}[Proof of Proposition~\ref{prop:counterexsc}]	
	Since $\omega(SC_{k+1})=k+1$, by Lemma~\ref{lem:counterex3}, $SC_{k+1}$ has strongly chordal treewidth $k$.
	By Lemma~\ref{lem:counterex2}, $SC_{k+1}/w_1w_4$ has strongly chordal treewidth at least $k+1$.
	Therefore, the class of graphs of strongly chordal treewidth at most $k$ is not closed under taking minors.
	\end{proof}

\section{Conclusions}
\label{section:conclusions}

In this paper, we consider the  graphs of special
treewidth, spaghetti treewidth, directed spaghetti treewidth, or
strongly chordal treewidth two. Similar to treewidth, pathwidth and treedepth,
these graph parameters can be defined as the minimum of the maximum clique size
over all supergraphs in some graph class $\cal G$, with $\cal G$ the class
of chordal graphs (in case of treewidth) or a subclass of the chordal graphs.
(See the discussion in Section~\ref{section:introduction} and Table~\ref{table:graphclasses}.)

Our main results are twofold: for each of the four parameters, we give the
obstruction set of the graphs with this parameter at most two. These obstruction
sets are summarized in Table~\ref{table:overview} in Section~\ref{section:introduction}.
Secondly, we give characterizations in terms of (special types of) trees of
cycles of the cell completion. A 2-connected graph has treewidth two, if and
only if its cell completion is a tree of cycles (see
Section~\ref{section:preliminaries}); for each of the other parameters, a similar
result with additional conditions on the tree of cycles exists. We summarize
these in Table~\ref{table:treemodels}. We have that the treewidth, spaghetti
treewidth, and strongly chordal treewidth of a graph equals the maximum of this
parameter over the blocks of the graph. This is not the case for pathwidth and
for special treewidth. For special treewidth two, we have established a precise
condition (building upon the notion of head vertices) how blocks of special
treewidth at most two can be connected to obtain a graph of special treewidth two
(see Section~\ref{section:special}).

\begin{table}[htb]
\begin{tabular}{|l|l|l|}
\hline
parameter & cycle tree model & connecting blocks \\
\hline
treewidth & tree of cycles \cite{BodlaenderK93,Kloks93} & everywhere \\
pathwidth & path of cycles \cite{Fluiter97,BodlaenderF96b} & !(not simple) \cite{Fluiter97} \\
spaghetti tw & chain tree of cycles (Th.~\ref{thm:main}) & everywhere \\
strongly chordal tw & tree of $2$-boundaried cycles & everywhere \\
& (Th.~\ref{thm:main2}) & \\
dir.~spaghetti & path of cycles (Cor.~\ref{cor:main6}) & everywhere\\
special tw & path of cycles (Cor.~\ref{cor:main6}) & head vertices \\
\hline
\end{tabular}
\caption{Models for cell completions of graphs with value of parameter at most two.
The second column gives the characterization for 2-connected graphs; the last column
shows how blocks of width at most 2 can be connected: everywhere = each block
has width at most 2 is sufficient; ! = no simple characterization exist; head vertices = see Section~\ref{section:charsp2}. }
\label{table:treemodels}
\end{table}

We expect that similar characterizations are hard or impossible to 
obtain for values
larger than two. For instance, we see in Section~\ref{section:width3} that
the classes of graphs of special treewidth, spaghetti treewidth, directed spaghetti
treewidth, or strongly chordal treewidth at most $k$, for $k\geq 3$ are not
closed under taking minors.

It may be interesting to pursue a similar investigation for parameters that
are defined in a similar way by other subclasses of chordal graphs. 
Bodlaender, Kratsch and Kreuzen~\cite{BodlaenderKK13} showed that special treewidth and
spaghetti treewidth are fixed parameter tractable; an adaptation of the
algorithms by Bodlaender and Kloks \cite{BodlaenderK96} or Lagergren and Arnborg
\cite{LagergrenA91} gives linear time decision algorithms for each fixed bound
on the width. We conjecture that in a similar way, it can be shown that
directed spaghetti treewidth is fixed parameter tractable. Whether strongly
chordal treewidth is fixed parameter tractable, we leave as an open problem.

\end{document}